\documentclass[12pt]{amsart}
\usepackage{amsfonts}
\usepackage{amsfonts,latexsym,rawfonts,amsmath,amssymb,amsthm}
\usepackage[plainpages=false]{hyperref}

\usepackage{graphicx}

\RequirePackage{color}

\numberwithin{equation}{section}

\def\p{\partial }

\def\M{\mathcal M}
\def\F{\mathcal F}

\def\R{\Bbb R}

\def\Om{\Omega}
\def\pom{\p  \Omega}

\def\I{\mathcal I}
\def\D{\mathcal D}
\def\A{\mathcal A}
\def\B{\mathcal B}
\def\P{\mathcal P}

\def\K{\mathcal K}

\def\Q{\mathcal Q}
\def\T{\mathcal T}

\newtheorem{Proposition}{Proposition}[section]
\newtheorem{Theorem}[Proposition]{Theorem}
\newtheorem{Lemma}[Proposition]{Lemma}
\newtheorem{Claim}[Proposition]{Claim}
\newtheorem{Corollary}[Proposition]{Corollary}
\newtheorem{Remark}[Proposition]{Remark}
\newtheorem{Example}[Proposition]{Example}
\newtheorem{Definition}[Proposition]{Definition}

\title  {  A Liouville theorem for the Neumann problem of Monge-Amp\`ere equations}\thanks{This work was supported by NSFC   11771237.}
\begin{document}

\address{Huaiyu Jian: Department of Mathematical Sciences, Tsinghua University, Beijing 100084, China.}

 \address{Xushan Tu: Department of Mathematical Sciences, Tsinghua University, Beijing 100084, China.}



\email{hjian@tsinghua.edu.cn; \;  txs17@mails.tsinghua.edu.cn. }


\bibliographystyle{plain}

\maketitle

\baselineskip=15.8pt
\parskip=3pt

\centerline {\bf  Huaiyu Jian}
\centerline {Department of Mathematical Sciences, Tsinghua University}
\centerline {Beijing 100084, China}

\vskip10pt

 \centerline { \bf Xushan Tu}
\centerline {Department of Mathematical Sciences, Tsinghua University}
\centerline {Beijing 100084, China}

  \vskip 15pt

\vskip20pt

\begin{abstract}
In this paper, we study  the Neumann problem of Monge-Amp\`ere equations in Semi-space.
For two dimensional case, we prove that  its viscosity convex solutions  must be a
quadratic  polynomial.
When the space dimension $n\geq 3$, we  show that the conclusion 
still holds if either the boundary value is zero or the   viscosity convex  solutions restricted on some $n-2$ dimensional subspace is bounded from above by a
quadratic function.
\end{abstract}

 \vskip20pt


\vskip20pt

\noindent {\bf  Running head}: Liouville theorem for   Monge-Amp\`ere equation

\vskip20pt

\baselineskip=15.8pt
\parskip=3pt

\newpage

\centerline {\bf   A Liouville theorem for the Neumann problem of Monge-Amp\`ere equations    }

 \vskip10pt

\centerline { Huaiyu Jian\ \ \  Xushan Tu  }

\maketitle

\baselineskip=15.8pt
\parskip=3.0pt

\section { Introduction}

The Bernstein theorem  for Monge-Amp\`ere equation $\operatorname{det}  D^2u(x)=1 $ in $\R ^n$ concludes that its solution must be a
quadratic  polynomial, which was proved in
\cite{[Jo], [Ca], [P1], [P2], [CY]}. This type of theorems  was extended to the Hessian equation\cite{[BCGJ]}, special Lagrangian equation\cite{[Y]}, affine maximal hypersurface
equation\cite{[TW]} and more general  Monge-Amp\'ere equation $\operatorname{det}  D^2u(x)=f(x) $ \cite{[CL1], [CL2], [CW], [JW3]}, and it is called Liouville theorem if additional 
assumption on the solution is necessary, just as the Liouville theorem for harmonic functions. 
The Liouville theorem  for  Monge-Amp\`ere equations in semi-space $\R_+ ^n$ with Dirichlet boundary value on  $\partial \R_+ ^n$ was studied in \cite{[JW2], [S1]}.

In this paper, we study the Liouville   theorem for the  Neumann
problem
\begin{equation} \label{liouville problem}
	\begin{cases}
		\operatorname{det}  D^2u(x)=1          & \text{ in }\R_+^{n},\\
		D_{n} u(x',0)= ax_1                       & \text{ on }\partial  \R_+^{n}.
	\end{cases}
\end{equation}
Here and below, $\R_+^n:=\{(x', x_n)\in \R^n:\; x_n>0\}$,
$D_n u=D_{e_n} u=\frac{\partial u}{\partial x_n}$, $a\in R$ is a constant, and $e_n$
is the standard unit vector on $x_n$-axis.

When $n\geq 3$, we need to assume that {\sl $u$ is quadratic growth in the auxiliary subspace at the infinite, i.e.,
\begin{equation} \label{growth condition}
	\lim_{z\in \R^{n-2}, z \to \infty}  \frac{u(0,z,0)}{|z|^2 } <\infty .
\end{equation} }
We use {\sl $ \Theta $ to denote the family of functions satisfying \eqref {growth condition} }.

We will use the standard   concept for (Alesandrov) generalized solutions and viscosity solutions to the first equation  of \eqref{liouville problem}. See the book \cite{[G]} for example.
Caffarelli\cite{[C1]} indicated that the generalized solution and viscosity solution to the first equation  of \eqref{liouville problem} are equivalent, which is proved in detail
  in  \cite{[G]}. This fact will be used again and again.   To describe the Neumann boundary value condition in a weak sense, we choose the viscosity solution.

  We say that  {\sl   $u$ is a viscosity subsolution ( or supersolution) of problem \eqref{liouville problem} if  $u \in C(\overline {\R_+^n})$  is a  convex function,
  and  for any $x_0\in\overline {\R_+^n}$ and any convex function $ \varphi \in  C^2(\overline {\R_+^n})$, the condition that $u-\varphi$ has
   a local maximun ( or minimum)
at $x_0$, implies
$$\begin{cases}
  \operatorname{det} D^2\varphi(x_0) \geq ( or\leq)  1       & if\; x_0\in   \R_+^n ,\\
D_n\varphi(x_0)\geq ( or\leq) \phi(x_0)                          & if\;  x_0\in \partial  \R_+^n .
\end{cases}$$
If   $u\in C(\overline {\R_+^n})$ is both  viscosity subsolution and viscosity supersolution, we call it a viscosity solution. }

The main result of this paper  is the following
\begin{Theorem} \label{liouville theorem}
	Let $u$  be a   viscosity solution to problem \eqref{liouville problem}.  If either $n=2$, or when $n\geq 3$,  $u \in \Theta $ or $a=0$,
	then $u$ is a quadratic  polynomial.
\end{Theorem}

In a future paper \cite{[JT]}, we will use Theorem \ref{liouville theorem} and perturbation method  to prove the global Schauder estimate for the viscosity solution to the Neumann problem of
 Monge-Amp\`ere equation $\operatorname{det}  D^2u(x)=f(x) $ in a bounded convex domain. Such estimate  was obtained in \cite{[S],[TW08]} for Dirichlet boundary value problem.

Observe that the condition
\eqref{growth condition} does not depend on the base point $0$, this is because
\[ 2u(a,b+z'',0) \leq u(0,2z'',0)+ u(2a,2b,0)\]
by the convexity, and \eqref{growth condition} is invariant after subtracting any linear functions.
Therefore, under a suitable linear transformation, our theorem actually applies to the following problem:
\begin{equation} \label{1.3 equation}
\begin{cases}
\operatorname{det}  D^2u(x)=c          & \text{ in }\R_+^{n},\\
D_{\alpha} u(x',0)= \beta \cdot x'  +b                    & \text{ on }\partial \R_+^{n},
\end{cases}
\end{equation}
where $c>0$ and $b\in R$ are constants, while $\alpha$ and $\beta$ are constant vectors such that $\alpha \cdot e_n >0$ and $\beta \cdot e_n =0$.
In this situation, the assumption \eqref{growth condition} changes to
\begin{equation} \label{1.4 equation}
	\lim_{z'\in H'',  z' \to \infty}  \frac{u(z',0)}{|z'|^2 } < \infty
\end{equation}
for some $n-2$ dimensional subspace $H''\subset R^n$ which is not parallel to
$\beta$. And Theorem 1.1 is turned to

\begin{Theorem} \label{a liouville theorem}
	Let $u \in C(\overline {\R_+^n})$ be a continuous, convex viscosity solution to problem \eqref{1.3 equation}. Suppose that either $n=2$, or when $n\geq 3$,  $u $ satisfies \eqref{1.4 equation} or $\beta=0$,
	then $u$ is a quadratic  polynomial.
\end{Theorem}

Our start point is to consider the level set at boundary point. First, we need to prove that $\nabla u(0)$ is well defined for a viscosity solution. Then the main task is
  to study the geometry  of  the height  section
\begin{equation} \label{section definition}
S_{h}(0) :=\{x \in \overline{\R_+^n} |\, u(x)<u(0)+\nabla u(0) \cdot x+h\}
\end{equation}
by blowing  down $u$  at the infinite point of  auxiliary subplane  and then blowing up at origin. As a result, we will finally prove that $D_nu(x)$
is a linear function.
The case $n=2$  or $a=0$  is  much more easy, which will be proved in Remark \ref{reflect classical results} and Remark 8.2 respectively.

This paper is arranged as follows.
In section 2, we will  show  the basic  properties for  the viscosity solutions of  problem \eqref{liouville problem}
 and  list a few  well-known Lemmas to be used.
In section 3, we will prove
 the strict convexity of solution to \eqref{liouville problem} away from the boundary  and then introduce the linearized problem for $D_nu(x)$
to obtain a boundary estimate of the normal derivatives of the solutions.
In section 4, we will study a general  global strict convexity lemma and provide an example showing that the growth assumption for the strict convexity is optimal.
In section 5,  we will prove  the height  section of solution to \eqref{liouville problem} has a good shape.  In section 6, we will use the result for case $a=0$ to prove a stability result for small $a$.
In section 7, we will prove the  global uniformly  convexity of solution to \eqref{liouville problem}.
In the last section, we will improve the famous  $C^{1,\alpha}$ estimate techniques\cite{[C1], [C3]},
 so as to prove that  the blow down-up limit has certain    $C^2$ regularity on a cone with vertex $0$,
 which implies that $u$ is a quadratic  polynomial.

\section{Preliminaries}
We start at  notations and definitions, then list some basic results. They will be used throughout this paper.

 A point  in $\R^{n}$ is written as
 $$
x=(x_1,\cdots,x_{n-1}, x_n)=(x',x_n)$$ and
as $x=(x_1,x'',x_n) $ if $n\geq 3$.

Let $\R_+^n= \{x\in \R^n|\, x_n>0\} $ be the upper half-space.  Then $\overline{\R_+^n}= \{x\in \R^n|\, x_n\geq 0\}$.  We will take the convention
$$R^{n-1}=\partial \R_+^n = \{x\in \R^n|\, x_n=0\}, \ \  R^{n-2}=\{x\in \R^n|\, x_1= x_n=0\}$$
from time to time.

 Denote by $\I$ (or $\I'$, $\I''$) the identity matrices of size $n$ (or $n-1$, $n-2$),  $B_r(x)$ (or $B_r'(x')$, $B_r''(x'')$) the ball of radius $r$ centered at $x$ (or $x'$, $x''$)
 in $\R^n$ (or $\R^{n-1}$, $\R^{n-2} $ respectively), $B_r^+(x)=\{y\in B_r(x):\; y_n>0\}$ and $\P x= x'$ is the projection mapping onto the $\R^{n-1}$.

The oscillation function of a function $W$ over a domain $E$ is denoted by
$$Osc_{E}W=\sup_{y,z\in  E}|W(y)-W(z)|.$$

Suppose that $u$ satisfies the assumption of Theorem \ref{liouville theorem}.  Since $u$ is continuous up to boundary,  we  have
\begin{equation}\label{section innitial assump}
	S_1(0) \supset B_{\kappa_0}^+(0)
\end{equation}
for some $\kappa_0>0$, where we have used notation \eqref{section definition} (Also see (2.4) below).

When $n\geq 3$,   the quadratic growth condition \eqref{growth condition} will be replaced by
\begin{equation}\label{growth innitial assump}
	u(0,x'',0) <\frac{K}{2}(|x''|^2+1) , \forall x''\in \R^{n-2}
\end{equation} for  a constant $K>0$.
We will  use {\sl  $\sigma_i:[0,\infty) \to [0,\infty)$,  $(i = 0, 1,2 )$, to denote some universal strictly increasing $C^0$ function with $\sigma_i(0)=0$ that depends only on $(n,a,\kappa_0, K)$. }

$\lambda $ and $ \Lambda $  will denote known positive constants.
  We always use the {\sl convention that $c$ or  $C$   denotes various positive constants   depending  only on
$n,a$,$\kappa_0, K $, $\lambda $ and $\Lambda$, } which are often called as  {\sl universal constants}.

 If not specified,  the $\operatorname{det}D^2 u$ will be understood as the general Monge-Amp\`ere measure of convex functions\cite{[F],[G]},
 which yields the (Alexandrov) generalized solution to the equation  $\operatorname{det}D^2 u=f$, and it is equivalent to
 the viscosity solution in the domain where $f$ is continuous.      If $u$ is a viscosity subsolution (or supersolution) of the equation
 $\operatorname{det}D^2 u=f $, we say that $u$ is a viscosity solution to the inequalities $\operatorname{det}D^2 u\geq (or \leq) f $,
 or $u$ solves these inequalities in viscosity senses.

The sub-differential of $u$ at   point $x_0$ is defined by
\[	\partial u(x_0) := \{ p|\, L:=u(x_0) + p \cdot (x-x_0)\text{ is a support plane of } u \text{ at } x_0 \}.\]
  We define $\nabla u(x_0)$ as a fixed element of $\partial u(x_0)$ for interior point. When $n=2$ or when $n\geq 3$ and  the assumption \eqref{growth condition} is satisfied, we will show  the viscosity solution to \eqref{liouville problem}, $u$,  is strictly convex
in $\R_+^n$. Then  the classical interior estimate \cite{[C1], [C3]} means that $u \in C_{loc}^{\infty}(\R_+^n)$ and $\nabla u$ is uniquely determined in $\R_+^n$.

 Take a  point $y=(y',0)\in \partial \R_+^n$.  If $u$ is convex, viscosity solution to  problem\eqref{liouville problem}, by the definition we have
\begin{equation} \label{viscosity sub equiv}
 \sup_{a\in \partial u(y) }\{ a\cdot e_n\} \leq  ay_1  ,
\end{equation}
and by   Lemma \ref{viscosity neumann equiv} in Section 3, one may define
\[  D_nu(y) = \varlimsup_{t\to 0^+}D_n u(y+te_n):=\lim_{t\to 0^+} \sup\{ p\cdot e_n:\; p\in  \partial u(y+te_n) \}. \]
And in Corollary \ref{en continues}, we will prove that there exists $p_y \in \partial u(y)$ such that
\[p \cdot e_n =  ay_1=D_n u(y).\]
Hence,  we define  {\sl $\nabla u(y)$ as a fixed $p_y$ for   any  boundary point $y \in \partial \R_+^n$}.

In this way, {\sl the section of $u$ centered at $y \in \overline{\R_+^n}$ with height $h>0$ } can be   defined by
\begin{equation}\label{2.5}
	S_{h}^{u}(y): =\left\{x \in \overline{\R_+^n}|\, u(x)<u(y)+\nabla u(y)\cdot(x-y)+h\right\},
\end{equation}
and its {\sl Neumann boundary } is  defined by
\begin{equation}\label{boundary innitial def}
	G_{h}^{u}(y) :=  S_{h}^{u}(y)  \cap  \partial \R_+^{n}.
\end{equation}
Sometimes, the notations $S_{h}^{u}(0)$, $G_{h}^{u}(0)$ , $B_r(0)$ and  $B_r^+(0)$ will be shorten as  $S_{h}$, $G_{h} $,
   $B_r$  and $B_r^+ $respectively.

 {\sl When we are going to prove Theorem \ref{liouville theorem}, we may  assume \begin{equation}\label{u=0 nabla u =0}
	 a\geq 0, \ \ u(0)=0 , \ \ \nabla u(0)=0 \; \text{ and }  u\geq 0\ \  \text {in }  	 \overline{\R_+^n}
	\end{equation}    by considering the function $u(x, -x_n)$ and subtracting the linear function $u(0)+  \nabla u(0)\cdot x $. }

\begin{Definition}[Good Shape]	\label{good shape definition}
We say that a section $S_{h}^{u}(0)$   is of good shape if  it satisfies the following two property:
	
	$S_{h}^{u}(0)$ has finite density at $0$, i.e.,
	\begin{equation*}
		c \leq \frac{\operatorname{Vol}(S_{h}^{u}({0}))}{h^{\frac{n}{2}} } \leq C,
	\end{equation*}
and
	\begin{equation*}
		c\P S_{h}^{u}({0}) \subset \P G_{h}^{u}(0) \cap (-\P G_{h}^{u}(0) ) .
	\end{equation*}
	\end{Definition}

When blowing  up at infinity, we will need to modify the class $\Theta$ to $\Theta(\delta)$.
\begin{Definition}[$\Theta( \delta)$ Class] \label{theta delta class}
	 Suppose that $n\geq 3$, $b\in R$ and $u$ satisfies
	\begin{equation*}
		D_n u(x',0)=bx_1, \forall x'\in \R^{n-1}
	\end{equation*}
	and
	\begin{equation*}
		u(0,z'',0)\leq  z''^T\Q z'' +\delta,  \forall x''\in \R^{n-2}
	\end{equation*}  for  some $( n-2) \times( n-2) $ symmetric matrix $\Q$  and $\delta >0$.
	We say $u \in\Theta( \delta)$ if the pair  $(b,\Q) $ satisfies
	\begin{equation*}
	 0 \leq b \leq c	\ \ \text{ and }  \ \  b^2 \operatorname{det}  \Q\leq C   .
	\end{equation*}
	\end{Definition}

\begin{Definition} \label{balance def}
	We say that a closed set $E\in \R^{n}$ is balance about $x$   up to a constant   $ \kappa $, if
	\begin{equation*}
		t(x-E)\subset E-x\;  \text{ for } \; t\leq c\kappa.
	\end{equation*}
	When $\kappa$ is universal, we say that $E\in \R^{n}$ is balance about $x$.
\end{Definition}

The following   lemmas will be used in our article.

Use $\operatorname{Vol}_k(E)$ to denote the $k-$dimensional Lebesgue measure and denote $\operatorname{Vol}(E)=\operatorname{Vol}_n(E)$.
\begin{Lemma} \label{balance kappa  lem}
	Given any convex set $E \in \R^n$.   Suppose the line $l(t): =x+ te_n $ intersects $E$ at the points $p,q$. Then
	\begin{equation}\label{quasi vol}
		|p_n-q_n| \cdot	\operatorname{Vol}_{n-1}(\P E)  \leq C\operatorname{Vol}(E).
	\end{equation}
	In addition, if $\P E$ is balance about $\P q$  up to  a constant $ \kappa $, then we have the  reverse inequality
	\begin{equation}\label{quasi k vol}
		|p_n-q_n| \cdot	\operatorname{Vol}_{n-1}(\P E)  \geq c\kappa \operatorname{Vol} (E).
	\end{equation}
\end{Lemma}
\begin{proof}
	Consider the non-negative concave  function defined on $\P E$
	\[ L(x')=\sup\{    (y-z)\cdot e_n|\,  y, z\in E,\  \text{ such that  }  \P y=\P z=x' \}.\]
	Then
	\[\operatorname{Vol} (E) \sim  ||L||_{L^{\infty}} \operatorname{Vol}_{n-1}(\P E),   \]
	which yields (2.7).   Now, suppose that $\P E$ is balance about $\P q$  up to the constant $ \kappa $. By concavity $L(\P q)  \geq c\kappa ||L||_{L^{\infty}}$, and \eqref{quasi k vol} follows.
\end{proof}

\begin{Lemma} \label{boundary nabla estimate}
	Let $u\in C(\overline{B_1^+(0) })$ be a convex function. Suppose that
	\begin{equation*}
		0\leq u \leq \sigma_1(|x|) \text{ in } B_c^+(0)
	\end{equation*}
	and
	\begin{equation*}
		\varlimsup_{t\to 0^+}D_n u(x',t) \geq -\sigma_0(|x'|), \forall (x' ,0)  \in  \overline{B_c^+(0)}.
	\end{equation*}
Then
\begin{equation}\label{boundary nabla estimate equation}
|\nabla u(y)| \leq C\left( \inf_{ |y| \leq t \leq c}\frac{2\sigma_1(1+t)}{t}+\sigma_0(|y'|)\right) \text{ in } B_c^+(0).
\end{equation}
 Here  we do not require $\sigma_i(0)=0\; (i=0, 1).$
\end{Lemma}
\begin{proof}
	Given $y \in B_{c}^+(0)$, $c \leq \frac{1}{2}$, let $p \in \partial u(y) $. The assumption and the convexity implies that
	\[	p_n    \geq -\sigma_0(|y'|).\]
	Suppose that $| p'|>0$. Note that  for every constant $a \in [0,1]$, the point $z=y+a| y|(\frac{p'}{|p'|},1)$ is in $ B_1^+(0)$.
Hence, by the convexity we have
	\[ a|y|(p_n+|p'| )=p\cdot (z-y) \leq u(z)-u(y) \leq 2\sigma_1((1+a)|y|).\]
   Letting $t=a|y|$,  we obtain the desired \eqref{boundary nabla estimate equation}.  We may
   ignore the estimation for $y\in \pom $, because $\nabla u(y)$ in this paper is determined by   taking the internal limit.	(See
   the above definition for $D_nu(y)$ and Corollary \ref{en continues}).
\end{proof}

Similar to Lemma \ref{boundary nabla estimate}, we can prove
\begin{Lemma}\label{interior nabla estimate}
	Write $x\in \R^n$ as $x=(x^1,\cdots,x^k)$, $x^j \in \R^{a_j}$, $\sum_{j=1}^{k} a_j=n$. Let $u\in C( \bar B_r(0) )$ be a convex function, satisfying
	\[u(x) \leq u(0)+ \sum_{i=1}^{k} \sigma_i(|x^i|).\]
	Then
	\[	|\nabla_{x_j}u(0)| \leq C \inf_{ 0 \leq t \leq cr}\frac{2\sigma_j(Ct)}{t}.\]
\end{Lemma}

\begin{Corollary} \label{small perturb means gradient lemma}
	Let $\epsilon>0$ be small constant,  $u\in C(\overline{B_1^+(0) })$ and $v \in C^2(\overline{B_1^+(0) })$ be two convex functions. Suppose that
	\begin{equation*}
		||u-v||_{L^{\infty}} \leq \epsilon
	\end{equation*}
	and
	\begin{equation*}
		\varlimsup_{t\to 0^+}D_n u(x',t) \geq D_nv(x',0)-\epsilon^{\frac{1}{2}}, \forall (x ,0)  \overline{B_1^+(0) } .
	\end{equation*}
	Then
	\begin{equation}\label{small perturb means gradient}
		|\nabla u(y)-\nabla v(y)| \leq  C_v\epsilon^{\frac{1}{2}}\text{ in } B_c^+(0),
	\end{equation} where $C_v$ is aconstant which may depends on $||v||_{C^2}$.
\end{Corollary}
\begin{proof}
	Given point $y \in B_c^+(0)$.  Consider the function
	 \[ U^y(x) = u(x) +2\epsilon-v(y) -\nabla v(y) \cdot (x-y).\]
	 Then by the assumption we have
	 \[  0\leq  U^y(x) \leq 2\epsilon+v(x) +2\epsilon-v(y) -\nabla v(y) \cdot (x-y) \leq C_v |x-y|^2+4\epsilon,\]
	 and
	 \[ \varlimsup_{t\to 0^+}D_n  U^y(x',t) \geq D_n v(x', 0) -\epsilon^{\frac{1}{2}}-D_nv(y',0)
\geq -C_v|x'-y'|-\epsilon^{\frac{1}{2}}.\]
	 Regarding  $y$ as origin and  applying Lemmas \ref{boundary nabla estimate} and \ref{interior nabla estimate} to $U^y(x)$, we get the desired results.
\end{proof}

The following  is well-known John's Lemma. See \cite{[F], [G],[John]} for example.
\begin{Lemma}[Jhon's Lemma] \label{Jhon Lemma}
If
 $\Omega $ is a bounded convex  set with non-empty interior in $\R^n$ and  $E$
 is the ellipsoid of smallest volume containing $\Omega , $ then after an affine transformation $T,$
 \begin{equation*}
 	B_{c(n)} \subset T(\Omega) \subset B_{C(n)}:=T(E).
 \end{equation*}
\end{Lemma}

\begin{Lemma} \label{comparison principle}
	If $\Omega $ is a bounded convex  set in $\R^n$ and $u, v\in C(\overline{\Omega})$  are convex functions in  $\Omega$  satisfying
\begin{equation*}
\begin{cases}
 \operatorname{det} D^2v \leq  \operatorname{det} D^2u        &  \text{ in } \Omega\\
 u\leq v                  & \text{ on }  \partial  \Omega \\
\end{cases},
 \end{equation*}
 then $u \leq v $ on $\overline{\Omega}$.
 \end{Lemma}
\begin{proof}
 See Theorem 2.10 in \cite{[F]}.
\end{proof}

\begin{Lemma}[ Comparison principle for mixed problem] \label{comparison principle for mixed problem}
 Supposed that $\Omega$ is a bounded convex set in $\R^n$,   $ G_1\subset\partial \Omega$
  and $G_2=\partial  \Omega  \backslash G_1  $ are piecewise $C^1$ smooth, and  $\beta$ is
     inward vector fieled on $ \partial \Omega_2 $.
     If $u, v$ are continuous, convex function defined on $\overline{\Omega}$ such that  either   $u$ or $v$ is $C^2$ smooth, satisfying
\begin{equation}\label{comparison problem 2.9}
\begin{cases}
\operatorname{det} D^2v < \operatorname{det} D^2u   <  \infty    & \text{ in } \Omega\\
u < v                   &\text{ on }     G_1 \\
 D_{\beta}v < D_{\beta}u                     & \text{ on }   G_2
\end{cases}
\end{equation}
 in the viscosity sense, then $u < v $ on $\overline{\Omega}$.
  \end{Lemma}
\begin{proof}
	On the contrary to the conclusion, we suppose that $w=u-v$ archives  positive maximum value $m$ at a point   $x_0 \in \overline{\Omega}$. According to the assumption and
  the first two inequalities in \eqref{comparison problem 2.9} we can take   $x_0\in \partial \Omega_2$  such that $u(x_0)-v(x_0)=m$.
  This means that  $u-[v(x_0)+\nabla v(x_0)\cdot (x-x_0)+m]$ archives
the maximum  at $x_0$ when
  $v$ is $C^2$ smooth,  or  $v-[u(x_0)+\nabla u(x_0)\cdot (x-x_0)-m]$
 archives  the  minimum  at $x_0$ when
  $u$ is $C^2$ smooth,   which contradicts  the third inequality in \eqref{u=0 nabla u =0} by the definition of viscosity solution.
\end{proof}

\begin{Lemma}\label{volume bound by measure}
  If $\Omega$ is a bounded convex domain in $\R^n$ and $u$ is a convex function satisfying
\[ \operatorname{det} D^2u > \lambda \text{ in }  \Omega , \]
then
\begin{equation} \label {2.11}
\operatorname{Vol}(\Omega) \leq C(\lambda)|| u||_{L^{\infty}}^{\frac{n}{2}} .
\end{equation}
 In particular,
\begin{equation} \label {2.12}
\operatorname{Vol}(S_h) \leq Ch^{\frac{n}{2}}.
\end{equation}
\end{Lemma}
\begin{proof}
\eqref{2.12} was proved in Corollary 3.2.5 in \cite{[G]}, and \eqref{2.11} can be proved as its proof.
\end{proof}

\begin{Lemma} \label{c1/n boundary estimate}
	 Suppose that $\Omega$ is a  convex domain  in $\R^n$  such that $B_{c}(0)\subset \Om  \subset B_{C}(0)$ and $u\in C(\overline{\Omega})$ is a convex function.
If
	\[ \lambda \leq \operatorname{det} D^2u \leq  \Lambda \text{ in }  \Omega, \ u=0 \text{ on } \partial \Omega ,\]
	then
	\begin{equation*}
		 u(x) \geq -C\operatorname{dist}(x,\partial \Omega)^{\frac{1}{n}}.
	\end{equation*}
\end{Lemma}
\begin{proof}
 This is the Aleksandrove's maximum principle. See Theorem 1.4.2   in \cite{[G]}.
\end{proof}

\begin{Lemma} \label{c2a classical estimate}
	Suppose that  $B_{c}(0)\subset \Om  \subset B_{C}(0)$ and $u$ is a convex function defined on convex domain $\bar \Omega$ such that
	\[ \lambda \leq \operatorname{det} D^2u =f(x)\leq  \Lambda \text{ in }  \Omega, \ u=0 \text{ on } \partial \Omega. \]
	If $||f||_{C^{k,\alpha}}(\Omega)\leq C$ for some $k \geq 0$ and $ \alpha\in (0,1)$, then
	\begin{equation*}
		||u||_{C^{k+2,\alpha}( \frac{1}{2}\Omega)} \leq C.
	\end{equation*}
\end{Lemma}
\begin{proof}
  When $k=0$, this is the famous interior $C^{2, \alpha}$ estimate in \cite{[C2]}. Also see \cite{[JW1]}.
  Then  the case of $k\geq 1$ follows from the standard regularity result for linear elliptic equations.
\end{proof}

\begin{Lemma} \label{strict classical line lemma}
 Let $u\in C(B_1(0) )$  be a convex   solution to
\[0 < \lambda \leq \operatorname{det}  D^2u \leq \Lambda .  \]
For $x\in B_1(0)$ and $p\in \partial u(x)$,  define
$  \Sigma:=\{y\in B_1(0)|\,u(y)=u(x)+p\cdot (y-x)\}.$
Then, either $\Sigma=\{x\}$ or $\Sigma $ has no extremal point inside $B_1(0)$.
\end{Lemma}
\begin{proof}
   See \cite{[C1]}.
\end{proof}

The following result is already well-known. See\cite{[Jo], [Ca], [P1], [P2], [CY]} or the books\cite{[F],[G]}.

\begin{Theorem} \label{global classical results}
	Let $u \in C({\R^n})$ be a continuous, convex   solution. If  $\operatorname{det} D^2 u =1$ in $\R^n$, then $u$ is a quadratic polynomial.
\end{Theorem}

\begin{Remark}\label{reflect classical results}
	   The case $a=0$ of Theorem \ref{liouville theorem}  can be obtained directly from Theorem \ref{global classical results}. In fact,  Lemma \ref{viscosity neumann equiv} in next section allow us to reflect $u$, the solution to problem \eqref{liouville problem},
with respect to $\partial \R_+^{n}$.
Then we  get a convex function $\bar{u}$.  Corollary \ref{en continues} also means that $D_n\bar{u}(\partial \R_+^n ) \subset \partial \R_+^{n}  . $
Thus $\bar{u}$ solves the equation
	\[	\operatorname{det} D^2 \bar{u} =1, \text{ in  } \; \R^{n}.\]
Hence $\bar u$ is a  quadratic polynomial by Theorem \ref{global classical results}, so is $u$.
\end{Remark}

\section{Interior strict convexity and normal derivative estimate  }

In this section, we first prove that a viscosity solution to problem \eqref{liouville problem}
is interior strict convexity, then give its normal derivative estimates.

\begin{Lemma} \label{strict convex in interior}
Suppose $u$ is a viscosity solution to problem \eqref{liouville problem}.
Then the graph of $u$ can't contains any ray, and $u$ is strictly convex away from the boundary $\partial \R_+^{n}$.
\end{Lemma}

\begin{proof}
If the graph of $u$ contains a ray $l$ and take $y \in l'=\P l $, the function
$$u^y(x)= u(x) -u(y)- \nabla u(y) \cdot (y-x) $$
 is bounded in the tubular domain $E:=(l'-y+B_1(y)) \cap \overline{\R_+^n}$ and
\[ \operatorname{det}  D^2u^y =1 \text{ in } E, \]
which contradicts  Lemma \ref{volume bound by measure} since $E$ has infinite volume.

Suppose $y  \in \R_+^{n} $. Let
\begin{equation} \label{maximum section heigh}
	  h_y^u := \sup\{ h  |\, S_{t}^{u}(y) \subset \R_+^n,\ \forall t < h\}.
\end{equation}
Since $u$ is strictly convex in the interior of sections\cite{[C1]},
  it is enough for the strict convexity of $u$ away from the boundary $\partial  \R_+^n$ to show $h_y^u >0$ for $y \in \R_+^{n}$.

Otherwise, if $h_y=0$, the convex set
\[\Gamma:=\{ x\in  \R_+^n |\, u(x)= u(y) +\nabla u(y) \cdot (x-y)\}\]
should contain more than one point. Lemma \ref{strict classical line lemma} implies that the extremal point of $\Gamma$ is not in interior.
There are two cases: either $\Gamma$ contains a line $l$ parallel to $\R^{n-1}$, or $\Gamma$ contains a ray $l$ starting from point $z=(z',0)$ on $\partial \R_+^{n}$.
 Both cases contradicts  the previous fact that the graph of $u$ can't contains any ray.
\end{proof}

\begin{Lemma} \label{viscosity neumann equiv}
Suppose $u$ is a viscosity solution to problem \eqref{liouville problem} and $y \in \partial \R_+ ^{n}$. Then
$D_nu(y):=\overline{\lim_{t\to 0^+}}D_n u(y', t)$ is well-defined. Moreover,
 \begin{equation} \label{viscosity sup equiv}
 	 D_n u(y)=ay_1
 \end{equation}
and
\begin{equation} \label{en growth neumann}
u(y+te_n) \geq  u (y)+ tay_1, \ \  \forall t >0.
\end{equation}
\end{Lemma}

\begin{proof}
Without loss of generality we  may assume $y=0$. By Lemma \ref{strict convex in interior} and the convexity of $u$, $D_n(te_n)$ are well-definend and continuous for $t>0$. Hence,
$$D_nu(0):=\overline{\lim_{t\to 0^+}}D_n u( te_n)= \lim_{t\to 0^+}D_n u( te_n)$$ is well-defined.
 Notice that $ay_1|_{y=0}=0$. Were if \eqref{viscosity sup equiv} false,  then
\[\lim_{t \to 0^+} D_n u (te_n)< -2\epsilon \]
for some $\epsilon >0$. Consider the tubular domain
\[ \Gamma= \{(x', 0)+st_0 e_n  |\, x'\in B'_r(0), \ s\in (0,1) \} \]
and choose the function
\[ v (x)= u(0)+\frac{  2\operatorname{Osc}_{B_r^+(0)}u}{r^2}|x'|^2   -\epsilon x_n   ,\]
where $t_0 $ and $r= r(\epsilon , t_0, u )$ are positive and so small  that
\[  D_nu(t_0 e_n )  <  -2\epsilon t_0,  \  \   2\operatorname{Osc}_{B_r^+(t_0e_n)}u \leq \epsilon t_0 \text{ and } |a|r  < \epsilon ,\]
which implies
\[   \sup_{y \in B_r(t_0e_n) \cup \partial  B_r^+(0)} (u(y)-v(y)) <0. \]
Notice that $v$ is linear along the $e_n$ direction and $u$ is convex, we find that
$v$ is strictly greater than  $u$ on the boundary $\partial \Gamma_1:= \overline{\partial \Gamma \setminus \R^{n-1}} $. Then, $v(0)=u(0)$
means that $v+c_0$ will touch $u$ by above at some point $z \in  \overline{\Gamma }\setminus \partial \Gamma_1$ for a proper $c_0 \in R$.
 Moreover, $\operatorname{det} D^2v=0 <1= \operatorname{det} D^2u $  implies $z \notin \Gamma $. Therefore, $z \in \partial \Gamma \cap R^{n-1}$. However,
\[D_{n}v(z) =  -\epsilon < -|a|r \leq az_1=D_{n}u(z), \]
which contradicts  the definition of viscosity supersolution. Thus, \eqref{viscosity sup equiv} is proved and \eqref{en growth neumann} follows by a direct integral.
\end{proof}

\begin{Corollary} \label{en continues}
	 If $u$ is a viscosity solution to problem \eqref{liouville problem} and $y \in \partial \R_+^{n}$, then there exists an element $p\in \partial u(y)$ such that $p\cdot e_n=D_n u(y)$.
\end{Corollary}
\begin{proof}
  Without loss of generality we  may assume $y=0$.   Let $p_t= \nabla u(te_n) $ and $l_t(x):=u(te_n) +p_t \cdot (x-te_n)$ be the support function
  of $u$ at $te_n$. Lemmas 3.2 and  2.5 implies that $p_t$ is uniformly bounded as $t\to 0^+$. Then, as $t\to 0^+$, $\{l_t \}$ is pre-compact  and contains
   a subsequence that converges to some support function $l(x):=u(0)+p \cdot  x $ of $u$ at $0$. Hence $p\in \partial u(0)$ and
     $D_n u(te_n) =p_t \cdot e_n$ monotonically decreases to $p\cdot e_n$.   Furthermore, by \eqref{viscosity sub equiv} and \eqref{viscosity sup equiv} we see that $p\cdot e_n=0=D_nu(0)$.
\end{proof}

\begin{Lemma} \label{linearized problem}
Suppose $u$ is a viscosity solution to problem \eqref{liouville problem}.
Then $D_n u \in C(\overline {\R_+^n})$  and $v(x):=D_nu(x)-D_nu(x',0)$ is a viscosity solution to following linearized problem:
\begin{equation} \label{linearized equation}
\begin{cases}
U^{ij} D_{ij}v=0        & \text{ in } {\R_+^n}\\
 v \geq 0      & \text{ on } \overline {\R_+^n}   \\
v= 0                       & \text{ on }  {\partial \R_+^{n}}
\end{cases},
\end{equation}
where $U^{ij}$ is the cofactor of the Hessian $D^2 u$.
\end{Lemma}

\begin{proof}
First, Lemmas $3.1$ and 2.13 means that $u\in C_{loc}^{\infty}({\R_+^n})$. Then it follows from the proof of Corollary \ref{en continues}
  that  the family of continuous function $f_t(x')=D_nu(x',t)$    point-wisely   decreases to  function $f_0(x') =D_nu(x',0)=ax_1$ monotonically
    as $t\to 0^+$. Finally, the finite cover Lemma  implies that this convergence is uniformly. Therefore $D_nu$ is continuous up to boundary.
	
The linearized equation $U^{ij} D_{ij}v=0  $  comes directly from the equation  $\operatorname{det}  D^2u =1$.
The other two  in \eqref{linearized equation} follows from the convexity of $u$ and  the continuity of $D_n u$ at the boundary.
\end{proof}

Next,  we introduce two functions,  $w^{y,+}(x)$ and $w^{y,-}(x)$, which will be used as the upper/below barriers of $v$ at $y$ to the linearized  problem \eqref{linearized equation}.

	Given $y\in  G_{\frac{1}{2}}^u(0)$ and consider the functions
	\begin{equation}\label{subtract support function}
		u^y(x)= u(x) -u(y)- \nabla u(y) \cdot (y-x)
	\end{equation}
	and
	\begin{equation} \label{test function}
		w^{y}(x) =u^y(x)-\frac{n}{2}x_nD_nu^y .
	\end{equation}
{ \sl Define
\begin{equation} \label{sup sub test function}
\begin{cases}
 w^{y,+}(x) =C_1[	w^{y}(x) +C_2x_n]\\
 w^{y,-}(x) =c_1[x_n-c_2	w^{y}(x) ]
\end{cases}.
\end{equation}
}
	Since $\operatorname{det} D^2u=1$,     we have $ U^{ij}D_{ij} D_n u=0$ in $\R_+^n$. Hence
\begin{equation} \label{sup sub test equationn}
 U^{ij}D_{ij}w^{y}=U^{ij}D_{ij}w^{0}= n \operatorname{det} D^2u-\frac{n}{2}(U^{in}D_{in}u+U^{ni}D_{ni}u)=0
\; \text{in}\;   \R_+^n.
\end{equation}

\begin{Lemma} \label{sup solution to lip}
Suppose that  $u$ is a viscosity solution to problem \eqref{liouville problem} and \eqref{u=0 nabla u =0} is satisfied.
 If
\begin{equation*}
B_c^+(0) \subset S_1(0) \subset B_C^+(0)
\end{equation*}
and
\begin{equation}\label{3.9}
	||D_nu ||_{L^{\infty}(S_1(0))} \leq C,
\end{equation}
then, $w^{0,+}(x)$ defined in \eqref{sup sub test function} is the upper barrier of $v$ at $0$ to the linearized problem \eqref{linearized equation}
in $S_1(0)$. (See \eqref{sup comparion} below).  This implies
\begin{equation}\label{en sup quadratic}
	D_n u(te_n) \leq Ct \ \  \text{and}\ \  u(te_n) \leq Ct^2
\end{equation}
 for all $t\geq 0$ such that $te_n\in S_1(0)$.
\end{Lemma}

\begin{proof}
	It follows from \eqref{3.9}   that
$$v(x)=D_nu(x)-D_nu(x',0)\leq C, \ \
\forall x\in S_1(0)   $$
	and
$$u(x)-\frac{n}{2}x_nD_nu  \geq 1-Cx_n, \ \  \forall x\in \partial S_1(0) \setminus G_1(0)   .$$
In \eqref{sup sub test function} we choose $C_1$ and $C_2$ such that they are larger than the $C$.  Then, $w^{0,+}(x) \geq C_1 $
on $ \partial S_1(0) \setminus G_1(0) $  by   (3.6). This, together with    \eqref{subtract support function}-\eqref{sup sub test equationn}, we see that
\begin{equation} \label{sup comparion}
\begin{cases}
U^{ij} D_{ij}w^{0,+}=0        & \text{ in } S_1(0)\\
 w^{0,+} \geq v     & \text{ on } \partial S_1(0) \setminus G_1(0)  \\
w^{0,+}\geq 0 =v                     & \text{ on }  {G_1(0)}
\end{cases}.
\end{equation}
Hence, the comparison principle   implies that $ v(x) \leq w^{0,+}(x)$ in  $S_1(0)$,
which, together with the convexity, yields the desired \eqref{en sup quadratic}.
\end{proof}

\begin{Lemma} \label{sub solution by low bar}
		Under the assumption of Lemma \ref{sup solution to lip},  there exists  a small $\bar{\sigma}>0$ depending on $\sigma_0$ such that if
	\begin{equation}\label{low bar assump}
	v(x) \geq \sigma_0(x_n)-\bar{\sigma} \text{ in } S_1(0),
	\end{equation}
	then $w^{0,-}(x)$ defined in \eqref{sup sub test function} is the lower barrier of $v$ at $0$ to the linearized problem \eqref{linearized equation} in $S_1(0)$
(see \eqref{sub quadratic} below).
This implies  that
\begin{equation} \label{en sub quadratic}
 v(te_n) \geq ct\ \  and \ \ u(te_n) \geq ct^2
 \end{equation}
 for all $t\geq 0$ such that $te_n\in S_1(0)$.
\end{Lemma}

\begin{proof}
	Denote $K=Cn||D_nu ||_{L^{\infty}(S_1(0))}$. Assume that $\bar \sigma >0$ is small and satisfies
	\[  \sigma_0^{-1}(2\bar \sigma) \leq \frac{1}{2K}.\]
	Recall that $ \sigma_0$ denotes a strictly increasing and continuous function from $[0, +\infty )$ to $[0, +\infty )$ with $\sigma_0(0)=0$.
Let $c_0= \sigma_0^{-1}(2\bar{\sigma})$. By
	\eqref{low bar assump} and the convexity we have
	\[ v(x) \geq \max\{\sigma_0(x_n)-\bar{\sigma},0\} \geq \bar{\sigma} \chi_{\{x_n \geq c_0\}},\]
	where $\chi_E$ denotes the characteristic function of the set $E$, i.e., $\chi_E(x)=1$ if $x\in E$, otherwise $\chi_E(x)=0$.
Take  $c_1=\frac{\bar{\sigma}}{2c}$ and $c_2=2c_0$. We have
		\[w^{0}(x) \geq 1-Kx_n, \ \  on \ \ \partial S_1(0) \setminus G_1(0). \]
Hence on $\partial S_1(0) \setminus G_1(0)$,  we see that
\begin{align*}
	w^{0,-}(x)& \leq c_1[x_n+c_2(Kx_n-1)] \\
	& \leq c_1[(2x_n-c_2) \chi_{\{x_n \leq c_0\}}+2x_n\chi_{\{x_n \geq c_0\}}] \\
	& \leq 2c_1 C\chi_{\{x_n \geq c_0\}} \\
	& = \bar{\sigma} \chi_{\{x_n \geq c_0\}}  \leq v(x).
\end{align*}
In conclusion, we have
\begin{equation} \label{sub quadratic}
	\begin{cases}
		U^{ij} D_{ij}w^{0,-}=0        & \text{ in } S_1(0)\\
		w^{0,-} \leq v    & \text{ on } \partial S_1(0) \setminus G_1(0)  \\
		w^{0,-}\leq 0    =v                  & \text{ on }  {G_1(0)}
	\end{cases}.
\end{equation}
The comparison principle   implies that $ w^{0,-}(x) \leq v(x)$ in $S_1(0)$.  In particular,
  $$v(te_n)  \geq w^{0,-}(te_n) \geq  c_1(t- Cc_2t)\geq \frac{c_1}{2}t$$
   when $c_2=2\sigma_0^{-1}(2\bar \sigma)>0$ is small, which is equivalent to $\bar \sigma>0$ is small
   enough.   This proves \eqref{en sub quadratic}.
\end{proof}

\begin{Corollary} \label{lip unversal by sup}
	Under the assumption of Lemma \ref{sup solution to lip}, there exists a small $\rho >0$ such that if
	\begin{equation} \label{lip unversal by eq}
		\sup_{B_{\rho}'(0)} ||  u(y' ,0)|| \leq \sigma_0 (\rho) \rho,
	\end{equation}
	then for any $y' \in B_{\rho}'(0)$ and $y=(y, 0)$, the function  $w^{y,+}(x)$ defined in \eqref{sup sub test function} is the upper barrier of $v$ at $y$ to the
linearized problem \eqref{linearized equation}. This also implies that $u^y(y+te_n) \leq Ct^2$ for small $t>0$ if $u(y)=0$ and $\nabla u(y)=0$.
\end{Corollary}
\begin{proof}
	 If $\rho$ is small,  the assumption implies that
	\[	\sup_{B_{\rho}'(0)} ||\nabla u (y' ,0|| \leq \min\{ \sigma_0 (\rho) , C\rho\} <<\frac{1}{2}. \]
	Then, the section
	\[ S_{\frac{1}{2}}(y) \subset S_{1-C|\nabla u(y)|}(y) \subset  S_1(0). \]
	Regarding $y$ as $0$ in Lemma \ref{sup solution to lip}, we see that this corollary is a direct consequence of Lemma \ref{sup solution to lip}.
\end{proof}

Similarly, we have

\begin{Corollary}
	Under the assumption of Lemma \ref{sub solution by low bar}, there exists a small $\rho>0$ such that if \eqref{lip unversal by eq} holds,
	then for any $y' \in B_{\rho}'(0)$ and $y=(y, 0)$, the function  $w^{y,-}(x)$ defined in \eqref{sup sub test function} is the lower barrier of $v$ at $y$ to
the linearized problem \eqref{linearized equation}. This also implies $u^y(y+te_n) \geq ct^2$ for small $t>0$ if $u(y)=0$ and $\nabla u(y)=0$.
\end{Corollary}

\vspace{8pt}

\section{Strict Convexity Lemma and Examples}

Recalling that if $u\in C(\overline{B_1^+(0) })$  is a  convex function and satisfies
\begin{equation}\label{4.1}
	\lambda\leq \operatorname{det}D^2u \leq \Lambda \text{ in } B_1^+(0),
\end{equation}
  we can naturally {\sl define
$$ D_n u(x',0) =\varlimsup_{t\to0+}D_n u(x',t) , \ \ \forall x' \in  B'_1(0). $$}
See Lemma 3.2 for the proof.
 In this section, we first study the strict convexity for such functions. The we provide an example showing that the growth assumption
 for the strict convexity is optimal.

\begin{Lemma}\label{strict convex lemma half}
	Given positive constants $H, M_1$ and $ M_2$.  Suppose that $u\in C(\overline{B_1^+(0) })$  is a  convex  function satisfying \eqref{4.1} and
	\begin{equation*}
		   D_n u(x',0) \geq -M_2 \text{ in } B'_1(0) .
	\end{equation*}
 There exists $\sigma=\sigma (H,M_1,  M_2, \lambda,  \Lambda, n)>0 $ such that if $u$ satisfies
	\begin{equation*}
		0 \leq u( x_1, x'',0 ) \leq  \sigma +  M_1|x''|^2   \text{ in } B_1'(0),
	\end{equation*}
	 then
	\begin{equation*}
		\sup_{B_1^+(0) }u \geq H.
	\end{equation*}
\end{Lemma}

\begin{proof}
	Assume by way of contradiction that there exists $\sigma_k \to 0^+$ and convex functions $u_k$ defined on $B_1^+(0) \subset \R^n$, satisfying
	\[0 < \lambda \leq \operatorname{det}  D^2u_k \leq \Lambda  \text{ in } B_1^+(0) , \     D_n u_k(x',0) \geq -M_2  \text{ in } B'_1(0) ,\]
	\[
	0 \leq u_k(x_1, x'',0) \leq \sigma_k +  M_1|x''|^2  \text{ in } B'_1(0)  ,
	\]
	and
	\[
	\sup_{B_1^+(0) }u_k \leq H.
	\]
	Applying Lemma \ref{boundary nabla estimate}  we see  that $u_k$ is locally Lipschitz with
$$||u_k||_{Lip(\overline{B_{4c}^+(0) })} \leq C(H+M_2).$$  Hence,
\[u_k ( x ) \leq \sigma_k+ M_1|x''|^2  +C(H+M_2)|x_n| \text{ in } B_{4c}^+(0).\]
	Let $h_k= \sigma_k^{\frac{1}{2}}$ and $\D_k=\operatorname{diag}(1, h_k^{\frac{1}{2}}\I'', h_k) $. Consider the functions
	\[	v_k( x)= \frac{u_k( \D_k x)}{h_k},\ x\in \Omega_k:= \D_k^{-1}B_{4c}^+(0) .\]
    Then $0 < \lambda \leq \operatorname{det}  D^2v_k \leq \Lambda $ in $   \Omega_k $,
    \[ v_k ( x ) \leq \sigma_k^{\frac{1}{2}}+ M_1|x''|^2  +C(H+M_2)|x_n| \text{ in } \D_k^{-1} B_{1}^+(0) \]
    and \[    D_n v_k(x',0) \geq -M_2  \text{ on } \D_k^{-1} B'_{1}(0) . \]
	Similar to Lemma \ref{boundary nabla estimate},  for any given $R>0$, we can prove that
	\[
	||v_k||_{Lip (B_R(0)\cap \Omega_k)} \leq  C( H+M_1+M_2)R^2.
	\]
	Therefore, $v_k$ is locally uniformly Lipschitz, and there is a subsequence that locally uniformly converges to
a convex  function $v$ defined on $ \Om := [-4c,4c] \times \R^{n-1}$. Moreover, we have
	\begin{equation} \label{4.2}
		0 < \lambda \leq \operatorname{det}  D^2v \leq \Lambda,
	\end{equation}
	and
	\begin{equation} \label{4.3}
		  v( x ) \leq  M_1|x''|^2  +C(H+M_2)|x_n| \text{ in } B_{4c}^+(0) ,  \     D_n v(x',0) \geq -M_2  \text{ on } B_R'(0).
	\end{equation}
	Furthermore, \eqref{4.3} implies that
	\begin{equation}\label{4.4}
		 a(x_1 )= \lim_{t \to 0^+}\frac{ v (x_1, 0,t )}{t}
	\end{equation}
	is well-defined, and
	\[ -M_2 \leq a(x_1) \leq C( H+M_2).\]
	Again, we can consider the blow-up of $v$ along axis $e_1$ by taking
$$\bar{h}_k= 2^k, \;  \bar{\D}_k:=\operatorname{diag}(1, \bar{h}_k^{\frac{1}{2}}\I'', \bar{h}_k), \;
   \bar{v}_k( x)= \frac{v( \bar{\D}_k x)}{\bar{h}_k}, \; x\in\bar{\D}_k^{-1}\Om .$$
    The limit $\bar{v}$ still satisfies \eqref{4.2}, \eqref{4.3}. And \eqref{4.4} gives
	\[	\bar{v}( x_1, 0,x_n )=  a(x_1)x_n.\]

	 Obviously, $\nabla \bar v(e_n) \cdot e_n=a(x_1)$.  The non-negative function
	\[w(x)=\bar v(x)-\bar v(e_n) -\nabla \bar v(e_n)   \cdot x, \]
	  is  zero on the line $te_n$.
	This is, $w$ is  bounded on a  convex domain with infinite volume. However, $ 0 < \lambda \leq \operatorname{det}  D^2w \leq \Lambda$, which contradicts with Lemma 2.11.

	
\end{proof}

\begin{Lemma} \label{strict convex lemma plane}
	Given positive constants $H$ and $M_1 $. Suppose $u\in C(\overline{B_1(0) })$  is a  convex  function satisfying
	\begin{equation*}
		\lambda \leq \operatorname{det}  D^2u \leq \Lambda \text{ in } B_1(0).
	\end{equation*}
	There exists $\sigma=\sigma (H,M_1,  M_2, \lambda,  \Lambda, n)>0 $ such that if $u$ satisfies
	\begin{equation*}
		0 \leq u( x_1, x'',0 ) \leq  \sigma +  M_1|x''|^2   \text{ on } B_1'(0),
	\end{equation*}
	then
	\begin{equation*}
		\sup_{B_1(0) }u \geq H.
	\end{equation*}
\end{Lemma}
\begin{proof}
	Assume by way of contradiction that  $\sup_{B_1(0) }u < H$. Then,  the assumption of Lemma \ref{strict convex lemma plane} immediately implies that
 $u \in  Lip(B_c(0))$ and $|\nabla' u(0)| \leq C(\sigma+\sigma^{\frac{1}{2}})$( See Lemma  2.6). After subtracting the supporting function
 $u(0)+\nabla u(0) \cdot x$ at $0$, $u$ satisfies the assumption of Lemma \ref{strict convex lemma half} with $M_2=\sigma+\sigma^{\frac{1}{2}}$. Applying Lemma \ref{strict convex lemma half} we
 obtain a contradiction, completing the proof.
\end{proof}

We point that the constant $\sigma $ in the growth assumption of Lemma 4.1 and 4.2 can not be  allowed to be zero. Otherwise, the constant $H$ may tend to $\infty$, yielding a contradiction.
Moreover, the following example shows that the growth exponent $2$ is  optimal.

\begin{Example} 
	Suppose $n \geq 3$,  constants $1<a ,b <\infty$ and $\delta >0$ is small such that
 $$\frac{1}{a}=\frac{1}{2}+\frac{1}{n-2}\delta,\; 1-\delta =\frac{1}{b} .$$
	  Let
	  \[E_1:=\{x' \in \R^{n-1}: \; |x''|^a \geq |x_1|^b\} ,\ \  E_2:=\{x' \in \R^{n-1}: \;|x_1|^b \geq |x''|^a \} . \]
	 Consider the function
	\begin{equation}\label{4.8}
		W_{a,b}(x')=\begin{cases}  |x''|^{a}+|x''|^{a-\frac{2a}{b}}|x_1|^2,    &x\in E_1 \\
			\frac{2b+a-ab}{b}|x_1|^{b}+\frac{ab-a}{b}|x_1|^{b-\frac{2b}{a}}|x''|^2,& x\in E_2
		\end{cases}
	\end{equation}
    and
    \begin{equation}\label{4.9}
    	W(x)= (1+x_n^2)W_{a,b}(x').
    \end{equation}
    We will show that  there is a $\rho=\rho (\delta)>0$ such that $ W$ is a convex function  in
    $R^{n-1}\times (-\rho, \rho) )$ satisfying
    \begin{equation}\label{4.10}
    	c(\delta)\leq   \operatorname{det} D^2W \leq C(\delta), \ D_nW(x',0)=0 \; \text{in}\; R^{n-1}\times (-\rho , \rho ).
    \end{equation}
	\end{Example}

    \begin{proof}
       Notice that $W \in C^1(\R^n) \cap C^2(\R^n\setminus \{x'=0\})$,  $\nabla W (\{x'=0\}) \subset \{x'=0\}$ and the   n-dim Hausdroff measure
       of the set  $\nabla W (\{x'=0\})  $  is zero.  The second equation in \eqref{4.10} is obvious.
     Hence, it is enough to  show that $W$ is convex  and satisfies the first inequality in  \eqref{4.10} in the sense of (Alexandrov) generalized solution.
      For this purpose, we need to calculate $\operatorname{det} D^2 W$ at any point $x'\neq 0$.

     It is easy to see that
    \begin{equation*}
       D^2 W=	\left[
    \begin{array}{cc}
    	(1+x_n^2)D'^2 W_{a,b}& 2x_nD' W_{a,b} \\
    2x_nD' W_{a,b} & 2 W_{a,b}
    \end{array}
    \right].
    \end{equation*}
    So the convexity of $W$ and $\det D^2W$ is reduced to  the study of  the matrix
    \begin{equation*}
        \A=	\left[
    	\begin{array}{cc}
    		(1+x_n^2)D'^2 W_{a,b}-2W_{a,b}^{-2}x_n^2D' W_{a,b} \otimes D' W_{a,b} & 0 \\
    		0 & 2 W_{a,b}
    	\end{array}
    	\right].
    \end{equation*}

 First, we assume $x'\in E_1$. Then
  $$ |x''|^a \leq  W_{a,b} \leq  2 |x''|^a ,\ \ \forall x'\in E_1 . $$
  Set the matrix functions
     $$\K_1'(x'')=\operatorname{diag}\{  |x''|^{a-\frac{2a}{b}}, |x''|^{a-2}\I''  \} , \ \  \K_1 (x'')=\operatorname{diag}\{\K_1', |x''|^{a}\}.$$
      Then by the assumption on $a$ and $b$ we have
    \begin{equation} \label{4.11}
     \operatorname{det} \K_1 =|x''|^{2a(\frac{n}{2}-\frac{1}{b}-\frac{n-2}{a})}\equiv  1.
    \end{equation}
    Observing that $C(a, b)$ can be written as $C(\delta)$, we  have
    \[0 \leq 2W_{a,b}^{-2}D' W_{a,b} \otimes D' W_{a,b} \leq C(\delta) \K_1' ,\ \  D'^2 W_{a,b} \leq C(\delta) \K_1',\ \ \forall x'\in E_1 . \]
     Hence,  in $E_1$ we have
    \begin{equation} \label{4.12}    C(\delta) \K_1 \geq \A \geq \left[
    \begin{array}{cc}
    	(1+x_n^2)D'^2 W_{a,b}-  x_n^2 C(\delta)\K_1'  & 0 \\
    	0 &   C(\delta)|x''|^{a}
    \end{array}
    \right].
     \end{equation}
     Since  $|x_n |\leq \rho$ will be small, by \eqref{4.11}-\eqref{4.12} we can conclude the convexity of $W$ and
     the  first inequality in \eqref{4.10} after we prove
    \begin{equation}\label{4.13}
    	D'^2 W_{a,b} \geq \epsilon(\delta) \K_1', \ \  \forall x'\in E_1 .
    \end{equation}
    Denote
    \[\B_1=\sqrt{\K_1'},\  X=|x''|^{-1}x'', \ y= |x_1||x''|^{-\frac{a}{b}} \]
    and
    \[ m=a[\I''+(a-2)X\otimes X]+ (a-\frac{2a}{b})[\I''+(a-\frac{2a}{b}-2)X\otimes X] y.\]
     By a direct calculation we have
    \begin{equation}\label{4.14}\B_1^{-1}D'^2 W_{a,b} \B_1^{-1} =  	\left[
    \begin{array}{cc}
    	2& 2  (a-\frac{2a}{b}) yX\\
    	2  (a-\frac{2a}{b}) yX^{T} &  m
    \end{array}
\right].
\end{equation}
	 Observe that
	\begin{equation}\label{4.15}  m-2(a-\frac{2a}{b})^2 |y|^2 X\otimes X= f_1 \I''-f_2 X\otimes X
\end{equation}
	where, $ f_1= 2a-\frac{2a}{b}$ and
	\[		f_2 =2(a-\frac{2a}{b})^2y^2-(a-\frac{2a}{b})(a-\frac{2a}{b}-2)y^2-a(a-2).\]
	  Notice that
	 \begin{align*}
	 	2(a-\frac{2a}{b})^2-(a-\frac{2a}{b})(a-\frac{2a}{b}-2) &=(a-\frac{2a}{b})(a-\frac{2a}{b}+2) \\
	 	&= a (a-\frac{2a}{b})(\frac{1}{n-2}-1)\delta \\
	 	& \geq 0.
	 \end{align*}
	Obviously, $|y|\leq 1$ in $E_1$.  Recall $\delta $ is small, $ a\geq \frac{2a}{b}-a \geq \frac{7}{4} >\frac{3}{2}$, thus  we get
	  \[ f_1-f_2 \geq (a-\frac{1}{2})^2-(\frac{2a}{b}-a-\frac{1}{2})^2 \geq c[a-(\frac{2a}{b}-a)]\geq c\delta a.\]
	 Since
	\[ \I''- X\otimes X \geq 0,\]
	we  obtain
	 \[ m-2(a-\frac{2a}{b})^2 |y|^2 X\otimes X \geq ca\delta \I''.\]
	 Replacing this in \eqref{4.15} and using \eqref{4.13}-\eqref{4.14}, we finally obtain \eqref{4.13}.
	
Second, we assume $x\in E_2$. Then we have
$$c(\delta)|x_1|^b \leq  W_{a,b} \leq C(\delta)|x_1|^b .$$
    Denote $$
    \K_2'=\operatorname{diag}\{ |x_1|^{b-2}, |x_1|^{b-\frac{2b}{a}} \I'' \} \; \text{ and}\;
    \K_2 =\operatorname{diag}\{\K_2', |x_1|^{b}\}.$$ Then
    \[ \operatorname{det} \K_2 =|x_1|^{2b(\frac{n}{2}-\frac{1}{b}-\frac{n-2}{a})}\equiv 1.\]
    Similar to  the first case  in $E_1$,  it is enough to show that
    \begin{equation}\label{4.16}
    	D'^2 W_{a,b} \geq \epsilon(\delta) \K_2'\ \ in E_2 .
    \end{equation}
    Denote
    \[\B_2=\sqrt{\K_2'},\ Y  = |x''|^{-1} x'',\ y\ =|x_1|^{-\frac{b}{a}}|x''|\]
    and $$ d=(2-a\delta)b(b-1)+a\delta (b-\frac{2b}{a})(b-\frac{2b}{a}-1)y^2.$$
   By a direct calculation we obtain
    \begin{equation}\label{4.17}
     \B_2^{-1}D'^2 W_{a,b} \B_2^{-1} =  	\left[
    \begin{array}{cc}
    	d  & 2 a\delta (b-\frac{2b}{a}) yY\\
    	2a\delta (b-\frac{2b}{a}) yY & 2 a\delta \I''
    \end{array}
    \right] .
    \end{equation}
    Write
    \[ d-2a\delta(b-\frac{2b}{a})^2 y^2 |Y|^2=g_1-g_2 y^2 , \]
    where
    \begin{align*}
    	g_2 &=a\delta (b-\frac{2b}{a})(b-\frac{2b}{a}-1)-2a\delta(b-\frac{2b}{a})^2\\
    	&=\delta b^2(2-a)(\frac{2}{a}+\frac{1}{b}-1) \\
    	& \geq 0
    \end{align*}
    Obviously, $|y|\leq 1$ in $E_1$.  If $\delta >0 $ is small,  we have
    \begin{align*}
    	d-2a\delta(b-\frac{2b}{a})^2 y^2 |Y|^2 & \geq b^2 \delta[(2-a\delta)+a\frac{2\delta }{n-2} (1-\frac{2}{a}-\frac{1}{b})- 2a(\frac{2\delta }{n-2})^2  ] \\
    	& \geq cb^2 \delta,
    \end{align*}
   which, together with \eqref{4.17}, implies the desired \eqref{4.16}.
\end{proof}

\begin{Remark}  
		Suppose that $R>0$, $W$ is given by \eqref{4.8}-\eqref{4.9}, and consider the Dirichlet problem
	\begin{equation}\label{4.18}
		\begin{cases}
			\operatorname{det} D^2 W_R^+ =c(\delta)  & \text{ in } B_{R}'(0) \times [-\rho, \rho]\\
			W_R^+ =W  & \text{ on } \partial (B_{R}'(0) \times [-\rho, \rho]).
		\end{cases}.
	\end{equation}
It is known that there exists a convex solution $W_R^+ $,  which is still symmetric about $x_n=0$ since so is $W$. This is to say
\[D_nW_R^+(x',0)=0.\]
	Comparison principle (Lemma \ref{comparison principle}) implies
	\[W_R^+ \geq W \geq 0\]
	 if $c(\delta)>0$ is small.  The convexity means
	\[
	W_R^+(x',t) \leq \max \{W_R^+(x',\rho),   W_R^+(x',-\rho)\} \leq 4W(x',0),\ \ x\in B_{R}'(0), t\in   [-\rho, \rho].
	\]
	and it follows that \[ W_R^+(te_n)  =0. \]
	This mean $W_R^+$ has the same asymptotic behavior around $0$ as $W$.
	
\end{Remark}


\section{Good Shape Lemma} 

In this section, we assume that $u$ is a viscosity solution to problem \eqref{liouville problem}, satisfying \eqref{u=0 nabla u =0} and  the hypothesis of Theorem \ref{liouville theorem}.
We are going to prove that its height section $S_h:=S_h^u(0)$ (defined by \eqref{2.5}) is of good shape for  large  $h$,
then we will introduce the normalization family $(u_h, S_1^h)$  of $(u, S_h)$.

At first, we mention here that the good shape property is invariant under any linear transform of the form $\A=\operatorname{diag}\{\A',a_n\}$
with $c\leq a_n, \; \operatorname{det} \A' \leq C$.

Let $p_h$ be  the mass center of $S_h $,  and consider the {\sl sliding transformation}
\[\A_h^c x=x-  \sum_{i=1}^{n-1}\frac{p_h \cdot e_i}{p_h\cdot e_n} x_n  e_i.  \]
Take $y_h \in \overline{\P (\A_h^cS_h )}$ such that
  \[ |y_h \cdot e_1|=\sup\{ x\cdot e_1 |\, x\in \P (\A_h^cS_h )\}\]
and  consider another sliding transformation
\[\A_h x=x-  \sum_{i=2}^{n-1}\frac{y_h \cdot e_i}{y_h\cdot e_1} x_1  e_i . \]
 Apply John's lemma (Lemma \ref{Jhon Lemma}) to $\{x'' \in R^{n-2} |\, (x_1,x'', 0) \in \A_h \P (\A_h^cS_h )\}$,
  we can find a new coordinates $\{e_i\}_{i=1}^n$ which keeps $e_1$ and $e_n$ (or only $e_n$ for $\A_h^c $) invariant such that
\[		\D_h B_{1}(q_h) \subset \A_h  \A_h^c  S_{h}    \subset  \D_h B_{C (n)}(q_h) \subset \D_h B_{2C (n)}^+(0),\]
where $q_h= \A_h \A_h^c  p_h$, $\D_h :=\operatorname{diag}(d_{1}(h), \cdots, d_n(h))$ is the diagonal matrix satisfying
  $$ \Pi_{i=1}^n d_{i}(h)=\operatorname{det}\D_h=Vol(S_h).$$

It follows from Lemma \ref{volume bound by measure}  that $\operatorname{Vol} (S_{h}) \leq  Ch^{\frac{n}{2}}$.  Denote
$$\tau_h= \frac{(\operatorname{det}\D_{h})^{\frac{1}{n}}}{h^{\frac{1}{2}}}, \; \bar{\D_h}=\tau_h^{-1} \D_{h}:=
\operatorname{diag}(\bar{d}_{1}(h), \cdots, \bar{d}_n(h)).$$

We  first consider  the rescaling   $(u_h, S_1^h )$ of $(u, S_h)$ given by
\begin{equation} \label{normalization tau definition}
	{u}_{h}(x)= \frac{u( {\A}_{h} ^{-1} \bar{\D_h} x)}{h}, \ \ x \in S_1^h :=  \bar{\D_h}^{-1}{\A}_{h} S_{h},
\end{equation}
then we have
\begin{equation}  \label{nomailzaition slide en}
	{\A}_{h}^*   B_{\tau_h }(x_h) \subset S_1^h \subset  {\A}_{h}^*  B_{C (n)\tau_h }(x_h) \subset \{ x_n \leq 2\tau_n C(n)\}
\end{equation}
for some $x_h$, where
\[   A_h^*=  \bar{\D_h}^{-1}  \A_h \A_h^c \A_h^{-1}\bar{\D_h}  \]
is still a sliding transformation of the form $ A_h^*x=x-\sum_{i=1}^{n-1}\beta_h^i x_ne_i$.

We  use the notation $G_1^{h}  :=  S_1^h   \cap  \partial \R_+^{n}$ similar to \eqref{boundary innitial def}.  And we will simplify ${S}_t^{h}, {G}_t^{h}$ by ${S}_t,{G}_t$
when   the dependence on $h$ and ${u}_h$
is self-evident from the context.
In Section 7, we will use the corresponding notations $S_t^R, G_t^R$ for the $(u_R, S_1^R)$ as we just defined  $S_1^h ,G_1^h$ for $(u, S_1)$.

 Now, the Neumann problem \eqref{liouville problem}  becomes
\begin{equation} \label{nomailzaition tau equation}
	\begin{cases}
		\operatorname{det}D^2 {u}_h=1     &  \text{ in } {S}_1^h\\
		D_{n} {u}_h =a_h x_1               &  \text{ on }  {G}_1^h \\
		{u}_{h}=1                   &  \text{ on }  \partial {S}_1^h \setminus {G}_1^h
	\end{cases}
\end{equation}
where
\begin{equation} \label{nomailzaition tau neumann equation}
	a_h =\frac{\bar{d}_1(h)\bar{d}_n(h)}{h}a.
\end{equation}

Recall that if $n\geq 3, $  $u\in \Theta $ and \eqref{growth innitial assump} is satisfied. Let
\begin{equation*}
	\delta_h=\frac{K}{h}, \ \bar{\Q}_h=\frac{C\bar{\D_h}''^TI'' \bar{\D_h}''}{h}.
\end{equation*}
We have
\begin{equation} \label{nomailzaition tau growth equation}
	u_h(0,z,0) \leq  z^T \bar \Q_h z +\delta_h, \; \forall z\in R^{n-2}
\end{equation}
and
\begin{equation} \label{nomailzaition tau aQ equation}
	a_h^2 \operatorname{det}\Q_h = C^n \tau_h ^{-2n} \frac{(\operatorname{det}\D_{h})^2}{h^{n} }a^2 \leq  C .
\end{equation}
If $h \geq 2K$, \eqref{growth innitial assump} also implies   $d_i(h) \geq ch^{\frac{1}{2}}$ for $2 \leq i \leq n-1$. More precisely,
\[ \Pi_{i=2}^{n-1}\bar{d}_i(h) \geq C\tau_h^{-\frac{n-2}{2}}\operatorname{Vol}_{n-2}(S_h(0) \cap \R^{n-2} )\geq Ch^{\frac{n-2}{2}}.\]
Therefore $\operatorname{det}\Q_h \geq c  $ and $0 \leq a_h \leq C$. To conclude,  we have obtained that
 \begin{equation} \label{nomailzaition tau general assump}
 0 \leq a_h \leq C, \ \  u_h\in \Theta (\delta_h), \ \   \forall h \geq 2K .
\end{equation}

 The Good Shape Lemma is stated as follows.
\begin{Theorem}[Good Shape] \label{good shape theorem}
	For large $h$, the normalization $(u_h,S_1^h)$ given by \eqref{normalization tau definition} satisfies
	\begin{align}
		\label{good shape cover}	&  \  \P S_1^h(0) \subset   C\P  G_1^h(0) \text{ and } S_1^h(0)
		\subset   B_{C (n)\tau_h }(0),\\
		\label{good shape volume}	&  \  c \leq \operatorname{Vol}(S_1^h(0)) \leq C \;  \text{ and }  \; c\leq \tau_h \leq C ,	 \\
		\label{good shape symmetric}	&  \   \P  G_1^h(0) \subset -C\P  G_1^h(0).
	\end{align}
\end{Theorem}

We are going to prove the theorem  by the  coming  three lemmas.
\begin{Lemma} \label{gs cover lemma}
The normalization $(u_h,S_1^h)$ given by \eqref{normalization tau definition} satisfies \eqref{good shape cover} for large $h$. More precisely,
\begin{equation*}
 \P S_1^h(0) \subset (1+C|a_h|) \P G_1^h(0) \subset C\P G_1^h(0).
\end{equation*}
\end{Lemma}

\begin{proof}

If there exists an unit direction $e, e\bot e_n$ such that
\[
\frac{\sup \{s |\, se\in  \P S_1^h(0)\}}{\sup \{s |\, te\in  \P G_1^h(0), \  \forall t\in (0,s)\}}  \geq  M  >1.
\]

We assume that the maximum values are attained    at points $y$ and $z$ respectively:
$$ y=y_ee+y_ne_n\in \partial \overline{S_1^h(0)}, \ \   z=z_ee \in \partial \overline{G_1^h(0)}.$$
  It is clear that
  \begin{equation} \label{5.11}
  	  u(z)=u(y)=1, \ \   y_e \geq Mz_e, \ \  0< y_n \leq C,
  \end{equation} where we have used \eqref{nomailzaition slide en} for the last inequality.

 Since the transforms   $  \A_h^*$ leaves $\R^{n-1}$ invariant,   we have $|z| \leq C$.  Let
\[l_{y,0}=\{sy,\  s\in (0,1)\} ,\]
\[ l_{z}^{e_n}=\{z+t e_n, \ t\in [0,\infty)\} .\]
There exists a point $p= sy=z+ty_ne_n \in l_{y,0} \cap l_{z}^{e_n}  $ with $s,t\in (0,\frac{1}{M}]$. By convexity, \eqref{u=0 nabla u =0} and \eqref{5.11} we have
\[u(p) \leq s  u(y) +(1-s)u(0)  \leq  \frac{1}{M} .\]
However, Lemma \ref{viscosity neumann equiv}, \eqref{5.11}, the convexity again,  and the fact $|z| \leq C$    implies that
\[u(p) \geq u(z)+t  y_n D_{n}u(z) \geq 1+ta_hy_nz\cdot e_1 \geq 1- C t|a_h|y_n  \geq 1-\frac{C|a_h|}{M}.\]
These two inequalities mean that $M \leq 1+C|a_h| \leq C$ by \eqref{nomailzaition tau general assump},
which implies $S_1^h(0)
		\subset   B_{C (n)\tau_h }(0)$.
 Combining  this  and  (5.2) we have
\[S_1^h \subset  \{x_n \leq C\tau_h\} \cap\{ CG_1^h \times [0,\infty]\} \subset B_{C\tau_h }(0).\]
\end{proof}

\begin{Lemma} \label{gs volume lemma}
The normalization $(u_h,S_1^h)$  given by \eqref{normalization tau definition} satisfies \eqref{good shape volume} for large $h$.
\end{Lemma}

\begin{proof}
Since $\operatorname{det}\D_h=Vol(S_h)\le Ch^{\frac{n}{2}}$ by Lemma 2.11, we see that $\tau_h\leq C$  by its definition. Notice that
$$\tau_h^{ n} \sim \operatorname{Vol}(S_1^h(0))=  \tau_h^{ n}  \operatorname{\A _h}\sim \tau_h^{ n} .$$
It is enough for (5.9) to prove that  $\tau_h \geq c$.

 Consider the function  $\upsilon$ defined by
\[ \upsilon(x)=\tau_h^{-1}\left(\frac{| x|^2 }{8C_n^2}+\frac{x_n}{2C_n}\right). \]
If $\tau_h>0$ is small, by \eqref{nomailzaition tau equation}, \eqref{nomailzaition tau general assump} and \eqref{good shape cover} we  have
\[\upsilon \leq \tau_h^{-1}\left(\frac{\tau_h^2}{8}+ \frac{\tau_h}{2}\right)   < 1 \leq u_h  \;  \text{ on }   \partial S_1^h \setminus G_1^h , \]
\[\operatorname{det} D^2\upsilon \geq c\tau_h^{-n}>1  =\operatorname{det} D^2u_h\;    \text{ in } S_1^h,\]
and
\[D_{n}v\geq  \frac{1}{2C_n\tau_h} > C \geq D_{n}u_h\;   \text{ on }  G_1^h. \]
By comparison principle (Lemma \ref{comparison principle for mixed problem}) $v <u_h$ in  $S_1^h$ ,  which contradicts   $v(0)=0 = u_h(0)$.  Therefore $\tau_h \geq c$.\\
\end{proof}

\begin{Lemma} \label{gs sym lemma}
The normalization $(u_h,S_1^h)$ given by \eqref{normalization tau definition} satisfies \eqref{good shape symmetric} for large $h$.
\end{Lemma}

\begin{proof}
Lemmas 5.2 and 5.3 mean that
$\tau_h$ is between two positive universal constants, and
\begin{equation} \label{5.12}
	\P S_1^h \subset C \P G_1^h \text{ and } B_c(x_h) \subset S_1^h \subset B_C(0)
\end{equation} for some $x_h\in R_+^n$.
Therefore, we only need to show $\P G_1^h$ is balance about $0$, this is
\begin{equation}  \label{5.13} \P {G}_1^h \supset B_c'(0). 	
	\end{equation}
	
Let $M \geq 4n^{4n}C^{4n}$ is large universal, $\theta = M^{-\frac{1}{n+1}}$, $m=M^{-1}$ and 
$\delta = \frac{1}{ M^3 }$ is small enough.
     On the contrary to  (5.13), there  would exist a $p=(p_1, p'') \in \partial \P{G}_1^h  $ such that
\begin{equation}\label{point distant less than delta}
	|p|=\operatorname{dist}(0,p)=\operatorname{dist}(0,\partial \P {G}_1^h ) \leq \delta.
\end{equation}
Without loss of generality, we may assume that  $p_1 \geq 0$. There are two cases:

Case 1:\;   $    |a_h| \geq CM^{-1}$. Then  we have
$$ \max\{d_2(h), \cdots,  d_{n-1}(h) \}\leq CMh^{\frac{1}{2}},$$ which  means that
\[ B_{cM^{-1}}''(0) \subset \P {S}_1^h(0) \subset B_C'(0) .\]
By convexity and \eqref{point distant less than delta}, we get
\[  \sup \{x_1|   x \in {G}_1^h\}  \leq CM\delta \leq \frac{C}{M}.\]
Therefore, \eqref{nomailzaition tau general assump} and the Neumann condition in \eqref{nomailzaition tau equation} implies
\[	|D_nu_h(x',0) | \leq \frac{C}{M}, \text{ on } x_1\geq  -\frac{1}{M}.
\]

Case 2:\;   $    |a_h| \leq CM^{-1}$. Then
\[	|D_nu_h(x',0) | \leq \frac{C}{M}, \text{ on } {G}_1^h . 
\]
Next,  we consider a new  coordinate with  zero as based point, axis ${e}_1= p/|p|$ and   $e_i (i \geq 2) $ unchanged. For simplicity, we still use the notation $(e_1, e_2,\cdots , e_n)$.


In both cases, consider  the domain $E= \{x\in S_1^h| x_1 > -\frac{1}{M}\}$.  Then $ \partial E=E_1 \cup E_2\cup E_3$,
where
\[
\begin{array}{l}
	E_1=\overline{E}\cap \{ x_1 = -\frac{1}{M}\}, \\
	E_2=\overline{E}\cap G_1^h \cap \{ x_1 > -\frac{1}{M}\} ,\\
	E_3 =\overline{E}\cap ( \partial S_1^h \setminus G_1^h ).
\end{array}
\]
For the function
\[ \upsilon(x)=\left[Q(Mx_1) +\theta  \sum_{i=2}^{n}Q(\frac{x_i}{ C})  \right] \]
 where  $Q(t)= \frac{1}{4nC^2}(\frac{t^2}{2}+2t),  $  we
  claim that
\begin{equation}\label{gs sym com equ}
	\begin{cases}
		\operatorname{det} D^2 {u_h} < \operatorname{det} D^2\upsilon        & \text{ in }  E\\
		D_{-x}{u_h} < D_{-x}\upsilon                    & \text{ on } E_1 \\
		D_{n} {u_h} < D_{n}\upsilon                    & \text{ on } E_2 \\
		\upsilon <  {u_h}                   &\text{ on } E_3\\
	\end{cases} .
\end{equation}
  By comparison principle (Lemma \ref{comparison principle for mixed problem}) we have $ \upsilon < u_h$ in $E$. This would contradict  $\upsilon(0)= u_h(0)$ and hence \eqref{5.13} is proved.

 Finally, let us  verify \eqref{gs sym com equ}. First, by convexity and \eqref{u=0 nabla u =0} we see that
 $$D_{-x}{u_h}=-D {u_h}\cdot x \leq  u_h(0)-u_h(x) \leq 0  $$ holds in viscosity sense. Applying this, the results of Case 1 and 2, and the largeness of $M$,
we have:

in $E$,   $\operatorname{det} D^2 \upsilon \geq CM^2 \theta^{n-1}    >1=\operatorname{det} D^2 {u_h}$;

on $E_1$,  $D_{-x}\upsilon \geq c-\theta \sum_{i=2}^{n-1}C  >0   \geq D_{-x} u_h$;

on $E_2$,  $D_{n}\upsilon  \geq c\theta  >\frac{C}{M} \geq D_{n}u_h  $;

on $E_3$,   $\upsilon\leq  \frac{3}{4} <1={u_h}$.
\end{proof}

 In this way, we have proved Theorem 5.1.  Since     $c\leq \tau_h\leq C$, we can still take
	  $$\D_h =\operatorname{diag}\{ d_1(h), \cdots , d_n(h)\}$$
in \eqref{normalization tau definition} instead of $\bar \D_h$. This yields the definition of normalization family.

\begin{Definition}\label{normal definition}
The normalization family $(u_h,S_1^h) $ of $(u, S_h)$ is defined as
\begin{equation} \label{normal function}
	{u}_{h}(x)= \frac{u( {\D_h}{\A}_{h} ^{-1} x)}{h}, \ x \in S_1^h: =  {\A}_{h}{\D_h}^{-1}  S_{h}.
\end{equation}
\end{Definition}
By \eqref{section innitial assump}, \eqref{nomailzaition tau equation}-\eqref{good shape symmetric}, we have the following properties for large $h$:
{\sl
	 \begin{equation}  \label{normal section}
		  B_{c}^+(0) \subset S_1^h \subset    B_{C }(0) ,
	\end{equation}
 \begin{equation} \label{normal equation}
	\begin{cases}
		\operatorname{det}D^2 {u}_h=1     &  \text{ in } {S}_1^h\\
		D_{n} {u}_h =a_h x_1               &  \text{ on }  {G}_1^h \\
		{u}_{h}=1                   &  \text{ on }  \partial {S}_1^h \setminus {G}_1^h
	\end{cases}
\end{equation}
where
\begin{equation} \label{normal neumann}
	a_h =\frac{{d}_1(h){d}_n(h)}{h}a,
\end{equation}

\begin{equation} \label{normal growth}
	u_h(0,z,0) \leq  z^T \Q_h z +\delta_h
\end{equation}
where
\begin{equation}\label{normal Qdelta}
	\delta_h=\frac{K}{h}, \ \bar{\Q}=\frac{C{\D_h}''^TI'' {\D_h}''}{h},
\end{equation} and
\begin{equation} \label{normal genral assump}
	0\leq a_h \leq C\text{ and } a_h^2 \operatorname{det} \Q_h \leq C.
\end{equation}
 Moreover, Theorem 5.1 still holds for the normalization family $(u_h,S_1^h) $. }

\section{Stationary Theorem} 

 In this section, we study the normalization family $(u_h, S_1^h)$ given by \eqref{normal function} in the case $a_h$ is small.
 For simplify we omit $h$ and write it by $(u,S_1 )$.    This leads us to study the  standard model
\begin{equation} \label{6.1}
	\begin{cases}
		\operatorname{det}D^2 {u}=1     &  \text{ in } {S}_1\\
		D_{n} {u} =a x_1               &  \text{ on }  {G}_1 \\
		{u}=1                   &  \text{ on }  \partial {S}_1 \setminus {G}_1
	\end{cases}
\end{equation}
and
 \begin{equation}  \label{6.2}
 	 u(0)=0 , \ \nabla u(0)=0,\  u\geq 0 \; in \; S_1,  \   B_{c}^+(0) \subset S_1 \subset    B_{C }(0) .
 \end{equation}
For every $a\in R$, set
\begin{equation} \label{6.3}
	Q_{a} (x):= \frac{1}{2}\left[\sqrt{1+a^2}(x_1^2+x_n^2)+2ax_1x_n+\sum_{i=2}^{n-1}x_i^2\right].
\end{equation}
Obviously, $\operatorname{det}D^2 Q_{a} (x) \equiv 1 $ and $D_n Q_{a} (x', 0)= ax_1$.

We will prove the following
\begin{Theorem} \label{stationary thm}
 Suppose that $u\in (\overline{S_1})$ is a viscosity solution to problem \eqref{6.1} and \eqref{6.2} is satisfied.
 There exists universal constants $b_0>0$, $c_0>0$ such that if
\begin{equation} \label{6.4}
\begin{split}
 ||u-Q_a||_{L^{\infty}(S_1^{Q_a}(0))}+  ||D_nu-D_nQ_a||_{L^{\infty}(S_1^{Q_a}(0) )} \leq b_0,
\end{split}
\end{equation}
then $u \in C^{2,\alpha}(\overline{B_{c_0}^+(0)})$  for some $\alpha >0$ and
\[||u||_{C^{2,\alpha}(\overline{B_{c_0}^+(0)})} \leq C.\]
\end{Theorem}

With this theorem, we will also show
\begin{Theorem}\label{zero petrub neumann}
Suppose that $u\in (\overline{S_1})$ is a viscosity solution to problem \eqref{6.1} and \eqref{6.2} is satisfied.
There exists universal constants   $\delta_0>0$, $c_0>0$ such if $  a \in (- \delta_0,  \delta_0)$, then
then $u \in C^{2,\alpha}(\overline{B_{c_0}^+(0)})$  for some $\alpha >0$ and
\begin{equation} \label{6.6}
\begin{split}
||u||_{C^{2,\alpha}(\overline{B_{c_0}^+(0)})} \leq C.
\end{split}
\end{equation}
\end{Theorem}

Denote
\begin{align*}
			\F_{a}^* (x): =& \{  Q|\, Q(x)=\sum_{i, j=1}^na^{ij}x_ix_j \text{ is a quadratic polynomial,} \\
	& \text{ and } Q \text{ solves the first and second equations of problem }\eqref{6.1} \}
\end{align*}
 and
\[	\F_{a} (x):= \{ l(x')+Q(x)|\, Q \in \F_a^*,   l(x') \text{ is a linear function }, l(0)=0 \}.\]

Theorem \ref{stationary thm} relies on the following Lemma.
\begin{Lemma} \label{station lemma}
Let $u \in C(\overline{S_1(0)})$ be a viscosity solution to problem \eqref{6.1}, and satisfy  \eqref{6.2}.
For any small constant $\mu>0$, there exists a small $\epsilon_0=\epsilon_0(\mu)>0$ such that if
\begin{equation} \label{6.9}
 ||u-P||_{L^{\infty}(S_1)}  \leq  \epsilon  ,\ \   ||D_nu-D_nQ||_{L^{\infty}(S_1)}  \leq C \epsilon
\end{equation}
for some  $\epsilon \in (0,  \epsilon_0) $ and $P(x)=l(x') +Q(x) \in \F_{a}$, then
 \begin{equation} \label{6.10}
 |Dl|\leq C\epsilon^{\frac{1}{2}},
\end{equation}
and
\begin{equation} \label{6.11}
 |u-P^0| \leq C \epsilon  \mu^{\frac{3}{2}}, \ \   |D_nu-D_nQ^0|\leq  C\epsilon  \mu \text{ in } S_{\mu}^{Q^0}
\end{equation}
 for  some $P^0=l^0+ Q^0 \in \F_{a}$. Moreover,
\begin{equation} \label{6.12}
 |D^2Q-D^2Q^0|\leq 2\epsilon  \text{ in } S_{\mu}^{Q^0}.
\end{equation}
\end{Lemma}

\begin{proof}
	 By virtue of the fact $l(0)=0 $, \eqref{6.9} gives $|D l| \leq C\epsilon$. One can choose  $x=(x', 0)\in G_1$ such that
$$-l(x')=\frac{1}{n}|Dl||x'| \, \text{ and }\; |x'|\sim \epsilon^{\frac{1}{2}}.$$
It follows from \eqref{6.2} and \eqref{6.9}   that
	\begin{align*}
		\frac{1}{n} |Dl||x'| =-l(x') & \leq   u(x',0)-l(x')  \\
		& \leq  |u(x, 0)-l(x')-Q(x', 0))|  +|Q(x',0)|  \\
		&  \leq \epsilon+C|x'|^2 \leq C|x'|^2 .
	\end{align*}
	Therefore $|Dl|\leq C\epsilon^{\frac{1}{2}}$ and \eqref{6.10} holds.

In following discussion, we will replace $u$ by $u-l(x')$ and need only to consider the case $P=Q$ in \eqref{6.9}. \eqref{6.2} and \eqref{6.9} implies that $c\I \leq D^2 Q \leq C\I$. For simplicity,
we further assume
$Q=Q_a$
(see \eqref{6.3}). This is because   we can always find  a new coordinate with the aid of the sliding transform $\B x :=x-2D''Q^{-1}D_{1i}Q x_1e_i $, under which $Q$ may be transformed to $Q_a$.

 Let
\begin{equation} \label{6.13}
  {u}_0:= u -P=u-Q, \ v_{\varepsilon}:= \frac{ {u}_0}{ \varepsilon   }, \ \  \varepsilon >0 .
\end{equation}
In fact, we may assume $ ||u-Q||_{L^{\infty}(S_1)}>0$ and take  $\varepsilon =||u-Q||_{L^{\infty}(S_1)} $.
 We will show the system  $(v_{\varepsilon},D_nv_{\varepsilon})$ could be well approximated by    $(w,D_nw)$, a solution  of some linearized problem.
More precisely, let $F= S_{\frac{1}{2}}^{Q}(0)$ and
\[	Lw:= \sqrt{1+a^2}(w_{11}+w_{nn})+ a w_{1n} +\Delta''w ,\] where  $\Delta'' :=\sum_{i=2}^{n-1}\partial_{x_i}^2$.
We will show that as if $\varepsilon \to 0^+$, the corresponding $(v_{\varepsilon},D_nv_{\varepsilon})$ will have a subsequence
  converging uniformly to  $(w,D_nw)$, a solution of the problem:
\begin{equation} \label {6.15}
	Lw= 0 \text{ in }  F,\ w(0)=0 ,\ D_nw= 0 \text{ on } \{ x_n=0\} .
\end{equation}

 \textbf{Step 1}.  We first consider the convergence on compact subset $E=\{ x_n \geq C\epsilon^{\frac{1}{2}} \} \cap F$, where $\epsilon>0$
 is sufficiently small.
 For any point $y$ such that $B_{C\rho}(y ) \subset E $  for some  $\rho >>c\epsilon^{\frac{1}{2}}$.
  Then $B_{c\rho}(y) \subset  S_{\rho^2}^{Q}(y) \subset B_{C\rho}(y ) $ by  the fact $Q=Q_a$.
 It follows from (6.1) that
\begin{equation} \label {6.16}
	0 =  \operatorname{det} D^{2}{u}-\operatorname{det} D^{2} Q=Tr  ( AD^2({u}-Q ))= \varepsilon Tr  ( AD^2 v_\varepsilon ),
\end{equation}
 where
\[  A =[A_{ij} ]_{n\times n}:= \int_0^1 cof ( ( 1-t )D^2Q+tD^2{u} ) dt,\]
and $cof M$ denotes the cofactor matrix of $M$.

Since \eqref{6.9} gives us $ |u-Q|\leq  \epsilon$ in $S_1$,  we see that
$$B_{2c\rho}(y ) \subset S_{\rho^2/2}^u(y) \subset B_{C\rho}(y ).$$
Hence, the classical interior estimate  (Lemma \ref{c2a classical estimate} and its proof in [3]) means that
\[   ||{u}||_{C^{k}(B_{c\rho}(y))}  \leq C  \rho^{2-k}  \text{ and } c\I \leq  D^2{u}\leq C\I \text{ in } B_{c\rho}(y).\]
Therefore, the operator $\bar L  f:=A_{ij} f_{ij}$ is uniformly elliptic with smooth coefficients. By  \eqref{6.16}, we obtain
\begin{equation} \label {6.17}
	||{v}_\varepsilon||_{C^{k}(B_{c\rho}(y))}\leq C\rho^{-k},
\end{equation}
which means
\[	| D^{2}{u}- D^{2} Q| \leq C\varepsilon \rho^{-2} \text{ and } | A- D^{2} Q|\leq C\varepsilon \rho^{-2}  \; \text{in}  \; B_{c\rho}(y)) .\]
Let $ \varepsilon \to 0^+$, we have $(v_\varepsilon,D_n{v}_\varepsilon)$ converges to $(w,D_nw)$  in $ B_{c\rho}(y)$.
By (6.10), we see that
$$ w (0)=0, \ \  L  w=0 \; \text{ in }\;
  B_{c\rho}(y).$$

\textbf{Step 2}.  We gives the uniform control of $(v_\varepsilon, D_nv_\varepsilon)$ near $\{ x_n =0\}$, so as to   show
\begin{equation} \label{6.19}
 |D_n v_\varepsilon(x)| \leq C|x_n|  \text{ in } S_{\frac{1}{4}}(0)
\end{equation}
and
\begin{equation} \label{6.20}
 \operatorname{Osc}_{B_r(x)}v_\varepsilon \leq C \max\{r^{\frac{2}{3}},\epsilon\} \text{ for } x \in G_{\frac{1}{4}}.
\end{equation}

Given $y \in G_{\frac{1}{2}}(0) $.  It follows from  \eqref{6.9} and Corollary \ref{small perturb means gradient lemma}
 that
$$|\nabla u-\nabla P|\leq C\epsilon^{\frac{1}{2}}\; \text{in}\; S_{\frac{1}{2}}(y)$$
if $\epsilon$ is small enough. Hence,
	the function $u^y(x)$ defined in \eqref{subtract support function} still satisfies
\[  	u^{y}(x) \geq c \text{ on } \partial S_{\frac{1}{4}}(y). \]
As in Lemma \ref{linearized problem},   $V :=D_n{v}_\varepsilon$ satisfies
 \begin{equation} \label{6.22}
\begin{cases}
U^{ij} D_{ij}(\pm V) =0        & \text{ in } S_{\frac{1}{4}}(0)\\
 |V |\leq C    & \text{ on } \partial S_{\frac{1}{4}}(0) \setminus G_{\frac{1}{4}}(0)  \\
 \pm V=  0                      & \text{ on }  {G_{\frac{1}{4}}(0)}
\end{cases}.
\end{equation}
Notice that the second in \eqref{6.22} is equivalent to  $V\leq C$ or $-V\leq C$, which comes from \eqref{6.9}.
  Similar to Corollary \ref{lip unversal by sup}, we see that  the function $	w^{y,+}(x) =C_1[	w^{y}(x) +C_2x_n]$ defined in \eqref{sup sub test function} is a supersolution
  for $V$ and  $-V$ in\eqref{6.22} respectively.

 Hence  we get
\[		|D_nv_\varepsilon(x)|=|V(x)|\leq w^{(x',0),+}(x',x_n) \leq Cx_n,  \text{ in } S_{\frac{1}{4}}(0).\]
 This roves  \eqref{6.19}.

 To prove \eqref{6.20}, we observe that
\eqref{6.19} implies
\begin{equation} \label{6.23}
|v_\varepsilon (x',0)-v_\varepsilon(x',x_n)|  \leq Cx_n^2 \text{ in } S_{\frac{1}{4}}(0).
\end{equation}
  It is enough to  estimate $|v_\varepsilon(p)-v_\varepsilon(q)|$ for $p=(p',0), q=(q',0)$ in $G_{\frac{1}{2}}(0)$. Suppose that
  $$|p'-q'| =r ,  \ \, \rho=C\max\{r^{\frac{1}{3}},\epsilon^{\frac{1}{2}}\},  \ \,
 A=p+\rho e_n,\ \ B=q+\rho e_n.$$
  As $\rho \geq C\epsilon^{\frac{1}{2}}$, \eqref{6.17} implies
\[ |v_\varepsilon(A)-v_\varepsilon(B)| \leq \frac{r}{\rho},\]
which, together with \eqref{6.23},  gives
\[ |v_\varepsilon(p)-v_\varepsilon(q)| \leq  C\rho^2 +\frac{r}{\rho} \leq C\max\{r^{\frac{2}{3}},\epsilon\}.\]
In summary, we have proved \eqref{6.19} and \eqref{6.20}. By this and the result of Step 1  we conclude that {\sl  as $ \varepsilon \to 0^+$,
 $(v_\varepsilon, D_n v_\varepsilon)$   has a subsequence  converging to a solution $w$ of problem \eqref{6.15} in $S_{\frac{1}{4}}(0)$ uniformly.}

\textbf{Step 3}. The classical regular results  then shows the function $w \in C_{loc}^3  $.
Since  $w(0)=0$, $D_nw(0)=0$ and $D'D_{n}w(0)=0$ by \eqref{6.15}, there exists  a linear function
 \[l^0(x) =\sum_{i=1}^{n-1} a_ix_i\]
and  a quadratic function
\[R(x)= \sum_{1\leq i,j \leq n-1} a_{ij}x_ix_j+b_nx_n^2\]
such that
\[	|w(x) -l^0(x)- R(x)| \leq C|x|^3  \text{ and }	 |D_n w(x) -2b_nx_n | \leq C|x|^2.\]
Here, $a_i,a_{ij},b_n$ are bounded by a universal constant, and
\[\sqrt{1+a^2}(a_{11}+b_ n)+ \sum_{i=2}^{n-1}a_{ii}= 0\]
by \eqref{6.15}.
As $\varepsilon\to 0^+ $, the uniform convergence means
\[	|v_\varepsilon (x) -l^0(x)- R(x)| \leq \sigma+ C|x|^3  \text{ and }	 |D_n v_\varepsilon (x) -D_nR | \leq \sigma+ C|x|^2,\]
where $\sigma =\sigma( \varepsilon ) \to 0^+  $. If we let $\mu \geq \sigma^{\frac{1}{3}} +\varepsilon ^{\frac{1}{3}}$,  then
\begin{equation} \label {6.25}
	| u- \left[Q+\varepsilon (l^0 +R) \right]|   \leq C \varepsilon \mu^{\frac{3}{2}} \text{ in }  S_{\mu}^{Q},
\end{equation}
and
\begin{equation} \label {6.26}
	|D_n u -(D_nQ+  \varepsilon D_nR)|  \leq C  \varepsilon \mu \text{ in }  S_{\mu}^{Q}.
\end{equation}

\textbf{Step 4}. Let
\[f(t):=\operatorname{det} \left[D^{2} Q+\epsilon D^2R+t \I \right].\]

Notice that $|f(0)-1| \leq C\epsilon^2$,  $f'(t) \sim C$. The equation $f(t)=1$ is solvable. The solution $t_0$ satisfies
\begin{equation}\label{6.27}
	|t_0| \leq C\epsilon^2.
\end{equation}
Let
\begin{equation}\label{6.28}
	 Q^0 (x)= Q(x)+ \epsilon R(x) +\frac{t_0}{2}|x|^2, \ P^0=\epsilon l^0+Q^0.
\end{equation}
Then \[\operatorname{det} D^{2}P^0=1,\ D_nP^0=D_nQ= ax_1\text{ on } \{ x_n=0\}, \]
which means that $P^0 \in \F_{a}$.

Finally, we combine \eqref{6.25},  \eqref{6.26},  \eqref{6.27} and \eqref{6.28}  to obtain
\[ |{u}(x)-P^0(x)| \leq  C\epsilon^2 \mu +C \epsilon \mu^{\frac{3}{2}} \leq   C \epsilon \mu^{\frac{3}{2}}  \text{ in }  S_{\mu}^{Q^0}, \]
and
\[ |D_n{u}(x)-D_n P^0(x)| \leq  C\epsilon^2 \mu^{\frac{1}{2}} +C \epsilon \mu \leq   C \epsilon \mu  \text{ in }  S_{\mu}^{Q^{0}}. \]
Noticing that $S_{\mu/2}^{Q_0} \subset S_{\mu}^{Q} $ and replacing $\mu$ with $\mu/2$, we have proved the desired \eqref{6.11}, while
\eqref{6.12} follows directly from \eqref{6.28}.
\end{proof}

{\bf  Proof of Theorem \ref{stationary thm}: }
 Let $C, \mu, $ and $ \epsilon_0(\mu)$ are the same as  in Lemma \ref{station lemma}.
Choose $0<\alpha<1$
and $0< \mu \leq (4C)^{\frac{2}{\alpha-1}}$ . Let
  $$C_{\infty}= \max\{\Pi_{j=0}^{\infty}(1+4C^2\mu^{\frac{j\alpha}{2}}), e^{4C}\}.$$
   Observing that
\[ \log C+ \Sigma_{j=0}^{k}\log (1+4CC_{\infty}^2\epsilon_0\mu^{\frac{k\alpha}{2}-2}) \leq 2C+\epsilon_0C(\mu) C_{\infty}^3\leq \log C_{\infty} , \]
we can assume that $\epsilon_0$ is sufficiently small such that the sequence
\[ C_0=C, \ C_k =(1+4CC_{\infty}^2\epsilon_0\mu^{\frac{k\alpha}{2}-2}) C_{k-1}, k=1, 2, 3, \cdots \]
is bounded by $C_{\infty}$.

We will prove by mathematical induction that
there exists a sequence of constants $a_k$, transformations
\[\M_k=\operatorname{diag}\{ \M_k',\M_{k,nn}\}, \ \det  {\M_k} =1\]
(where $\M_k'$ is a $n-1$ square matrix),  and linear functions $l_k=\sum_{i=1}^{n-1} a_i^kx_i$
such that at height $h_k:=\mu ^k$,  the  normalization solution $(u_{k}, \Omega_{k} )$ of $(u, \Omega)$  given by
\begin{equation} \label{22} 	
u_{k}(x)= \frac{u( \T_k x)}{h_{k}},\ x \in  \Omega_k:= \T_k^{-1} S_{h_k}(0), \  \T_k = h_k^{\frac{1}{2}}\M_k
\end{equation}
satisfying problem (6.1) and
\begin{equation} \label{23}
 ||u_k-l_k-Q_{a_k}||_{L^{\infty}(S_1^{P_{k}})} +||D_nu_k-D_nQ_{a_k}||_{L^{\infty}(S_1^{P_{k}})} \leq \epsilon_k:= \epsilon_0 \mu^{\frac{{k\alpha}}{2}},
\end{equation}
for $P_k=l_k+Q_k \in \F_{a}$, where
\[ 	Q_k(x):= \frac{1}{2}\left[\sqrt{1+a_k^2}(x_1^2+x_2^2)+2a_kx_1x_2+\sum_{i=2}^{n-1}x_i^2\right].\]
Moreover,
 \begin{equation} \label{24}
\begin{cases}
& |b_{k-1}| \leq C(C_{\infty})\epsilon_{k-1}^{\frac{1}{2}}, \ \ |a_k| \leq C_k, \ \  |{\M}_k{\M}_{k-1}^{-1}-\I|\leq C_k\epsilon_{k-1}\\
 &  C_k^{-1}B_1(0) \cap \Omega_k \subset S_1^{u_k}(0) \subset C_kB_1(0).
\end{cases}.
\end{equation}

The case $k=0$ is the assumption of Theorem 6.1. If the case of $k$ holds we will prove it for the case $k+1$. We check for $ \epsilon_k$  whether $u_k $ satisfies the conditions in
Lemma 6.3. By induction hypothesis (6.23), we have
\[	|u_{k}-P_k| \leq \epsilon_{k} \text{ in } \Omega_{k} \cap S_1^{P_k}.\]
Apply  lemma 6.3  to $u_k$, we obtain the quadratic function $  \bar P_{k+1}=\bar l_{k+1}+\bar Q_{k+1} \in \F_{a}$ satisfies
 \begin{equation} \label{25}
\begin{cases}
& \operatorname{dist}_{\mu}(u_k,\bar P_{k+1},0)  \leq C{\epsilon}_k \mu^{\frac{3}{2}} \leq \epsilon_{k+1} \mu, \\
 &  |b_k| \leq C\epsilon_{k}^{\frac{1}{2}} \;  \text{ and } \;  |D^2Q_k-D^2 \bar Q_{k+1}|\leq C{\epsilon}_k  .
 \end{cases}.
\end{equation}
Then there exists $a_k$ and  positive  matrix $\B_{k+1} =\operatorname{diag}\{ \B_k',\B_{k,nn}\}$ satisfies
\[\B_{k+1}^TD^2Q_k \B_{k+1} = D^2 Q_{k+1}\]
and
\[ |a_{k+1} -a_k| \leq CC_k \epsilon_k  \text{ and } |\B_{k+1}-\I| \leq CC_{k}{\epsilon}_k .\]
In particular,  $|a_{k+1}| \leq  C_{k+1}$.
Take  $\M_{k+1}=\M_k\B_{k+1}$, $\T_{k+1}= \mu^{\frac{1}{2}}\T_k\B_{k+1}$ and let
\[	u_{k+1}(x)= \frac{u({\T}_{k+1} x)}{h_{k+1}},\  S_1^{u_{k+1}}(0)=  \T_{k+1}^{-1} S_{h_{k+1}}(0), \ h_{k+1}= \mu^{{(k+1)\alpha}}.\]
Then we have
\[ 	(1-CC_{\infty}^2\mu^{-1}\epsilon_{k})C_k^{-1} B_1(0) \cap \Omega_{k+1}
\subset S_1^{u_{k+1}}(0)
\subset  (1+CC_{\infty}^2\mu^{-1}\epsilon_{k})C_k B_1(0) ,\]
which is
\[	C_{k+1}^{-1}B_1(0) \cap \Omega_{k+1} \subset S_1^{u_{k+1}}(0) \subset C_{k+1}B_1(0).\]
 Using the above facts and (6.25), one may verify by direct calculations that (6.23)-(6.24) holds for $k+1$.

It follows that $|a_k| \leq C_k$, $|\M_k| \leq C_k$  converge geometrically to $a_{\infty}$, $\M_{\infty} $ and
\[ 	|\M_{\infty}\M_{k}^{-1}-\I| \leq C_k \epsilon_{k} \leq C_{\infty}h_k^{\frac{{\alpha}}{2}}.\]
We replace each $\M_{k}$ by $\M_{\infty}$ and $Q_k $ by $ Q_{\infty}$, the inductive hypothesis still holds with a large universal constant. This is
\[	|u(x ) -h_k^{\frac{1}{2}} l_k(\M_{\infty}^{-1} x)-Q_{\infty}(\M_{\infty}^{-1}x )|  \leq C_1C_{\infty}^2h_k^{1+\frac{{\alpha}}{2}} \text{ in } B_{ch_k^{\frac{1}{2}}}(0).\]
 Recalling that $|Dl_k|\leq C\epsilon_{k-1}^{\frac{1}{2}}$, we obtain
\[	|u(x ) -Q_{\infty}(\M_{\infty}^{-1}x )| \leq  C_1C_{\infty}^2h_k^{1+\frac{{\alpha}}{2}}  \text{ in } B_{ch_k^{\frac{1}{2}}}(0),\]
which  means
\begin{equation} \label {26}
	|u(x ) -Q_{\infty}(\M_{\infty}^{-1}x )| \leq  C_1C_{\infty}^2|x|^{2+\alpha } .
\end{equation}

The above induction argument   and Lemma 6.3 could be applies to any point $y\in G_{c_0}(0)$. As a result, we may obtain that
\begin{equation*}
 |u(y+x ) -Q^y( x )| \leq  C|x|^{2+{\alpha}}
\end{equation*}
for a quadratic polynomial depending on the based point $y$.  In fact, consider  $p=y+\theta e_n$    with $\theta >0$ is small enough.
 The height section $S_{c_1\theta^2}^{Q^y}( p) $ is contained in $S_1$ with
\[  B_{c\theta^2}(p )  \subset S_{c_1\theta^2}^{Q^y}( p)  \subset B_{C\theta^2}(p ). \]
 By (6.26) we have
$$\begin{cases}
	\operatorname{det} D^{2}u (x )  = 1      & \text{ in } S_{c_1\theta^2}^{Q^y}( p)\\
	|u-Q^y(x ) |\leq C|\theta|^{2+\alpha}                  & \text{ on } \partial S_{c_1\theta^2}^{Q^y}( p )
\end{cases}. $$
By interior $C^{2,\alpha}$  theory and its perturbation method (See [3] for the details), $ ||u||_{C^{2,\alpha}(B_{c\theta^2}(p ))}\leq C$.
Since the constant $C$ is independent of $y$ and $\theta$,      $u \in C^{2,\alpha}(\overline{B_{c_0}^+(0)})$ for  a small $c_0>0$.

\vskip 0.5cm
{\bf  Proof of Theorem \ref{zero petrub neumann}: }
Choose $\delta_0=b_0^{6n}$ and $c_0= cb_0$, where $\alpha, \mu, $ and $b_0$ are the same as in the proof of Theorem \ref{stationary thm}.
Take $a\in (-\delta_0,   \delta_0)$. We are going to  construct a quadratic approximation of $u $ near $ 0 $ as \eqref{6.4}, and  applying
Theorem \ref{stationary thm} to complete the proof of Theorem \ref{zero petrub neumann}.

Lemma \ref{gs cover lemma}  implies
\[  \P S_1(0)   \subset  (1+C\delta_0)G_1(0).\]
For $(x', 0)\in PS_1(0)$, let $g( x' )$ be the maximum/minimum value such that $( x',g( x' ) ) \in S_1(0)$. We consider the convex set
\[ E=\{ ( x' , x_n ) |  (x', 0)\in PS_1(0), 0 \leq x_n \leq   g( x' )\},\]
and
\[ S=\{ ( x' , \pm x_n ) | ( x' , x_n ) \in E\}.\]
Then
\[\partial S_1 \setminus G_1 \subset E_{C\delta_0} =\{x \in \R ^n| \operatorname{dist}( x, \partial E )  \leq C\delta_0 \}. \]
 By the standard theory(see \cite{[F], [G]} for example) we see that
there is  a convex function $v$ symmetric about the hyperplane  $\{x_n=0\}$  solving
\[\begin{cases}
	\operatorname{det}  D^2 v = 1     & \text{ in }  S\\
	v= 1                  & \text{ on } \partial S \\
\end{cases}.\]
Moreover, by \eqref{6.2} and Lemma \ref{c2a classical estimate}, $v$ is smooth in the interior and Lemma \ref{c1/n boundary estimate} tells us
\[ v( x ) \geq  1-C\operatorname{dist}( x, \partial S )^{\frac{1}{n}}.\]
Thus, $v$ is also a solution to the following Neumann problem
\begin{equation}  \label{6.38}
\begin{cases}
\operatorname{det}  D^2 v = 1       & \text{ in }  S_1\\
1-b_0^6 \leq v \leq 1                  & \text{ on } \partial S_1 \setminus G_1  \\
D_nv= 0                 & \text{ on } G_1\\
\end{cases}.
\end{equation}
Since $|a|<\delta_0$ and $b_0>0$ is small,
then $v^{\pm}:=v\pm 2b_0^6(2C-x_n)$ must be supersolution/subsolution of $u$ to \eqref{6.1}, which implies
\begin{equation}  \label{6.39}
 ||v-u||_{L^{\infty}(S_1)} \leq   C b_0^6
\end{equation} by the comparison principle (Lemma \ref{comparison principle for mixed problem}).
Let $w=u-v$, by Corollary \ref{small perturb means gradient lemma} we have
\begin{equation}  \label{6.40}
 |\nabla w| \leq Cb_0^{3} \text{ in } B_{c}^+(0).
\end{equation}
%

By \eqref{6.38} we have $D_{in}v=0$ for $i\neq n$ in $G_1$. Let
\[\M =  D_{ij}v(0),  \ \  R(x)=x^T\M x , \ \ \M=\B^2,\]
where $\B=(b_{ij})_{n\times n}$  may be chosen such that it is bounded symmetric matrix with $b_{in}=0$ for $i\neq n$
because of the same property of $\M$. Consider the diagonal matrices $\B_1$ defined by
$$ \B_1^2=[\frac{1}{2}diag\{\sqrt{1+a^2}, I'' , \sqrt{1+a^2}\}]^{-1}$$ and the function
\[
\tilde{u}(x) = \frac{   u(b_0 (\B\B^{-1})x)}{b_0^2}, \\ Q(x)=\frac{Q_a(b_0 (\B \B_1)^{-1}x)}{b_0^2}, \ x   \in B_C(0),\]
where $Q_a$ is the same as in (6.3).
Letting $E_0= B_{Cb_0}(0)$ and noticing   $v \in C^3(B_c(0))$, we see that
\begin{align*}
	 || \tilde{u} -Q||_{L^{\infty}(E_0)} & \leq  \frac{1}{b_0^2}||u-R||_{L^{\infty}(E_0)}
	+C|a| \\
	& \leq  \frac{1}{b_0^2}||u-v||_{L^{\infty}(E_0)}+Cb_0 +C|a| \\
	& \leq Cb_0
\end{align*}
and
\begin{align*}
	||D_n \tilde{u} - D_n Q||_ {L^{\infty}(E_0)} & \leq \frac{1}{b_0}||D_nu-x_nD_{nn}v||_{L^{\infty}(E_0)}+C|a|  \\
	& \leq \frac{1}{b_0}||D_nu-D_n v||_{L^{\infty}(E_0)}+Cb_0+C|a|  \\
	& \leq Cb_0.
\end{align*}
By Theorem \ref{stationary thm},  we have $\tilde{u} \in C^{2,\alpha}(\overline{B_c^+(0)}) $, therefore $u\in C^{2,\alpha}(\overline{B_{c_0}^+(0)})$ for $c_0=cb_0$.

\section{Universal Strict Convexity}\

Throughout this section, we assume that $u$ is a viscosity solution to problem \eqref{liouville problem}, satisfying \eqref{growth condition} (or \eqref{growth innitial assump}), \eqref{section innitial assump} and \eqref{u=0 nabla u =0}.
By \eqref{section innitial assump} and Theorem \ref{good shape theorem} we may assume
\begin{equation}\label{7.1}
	B_{c}^+(0) \subset S_1(0) \subset B_{C}(0).
\end{equation}
  We will study the geometry of  the section $S_R(0)$  for large $R$.  In particular, we will show the uniformly strict convexity for the normalization family $(u_R,S_1^R)$
  defined by \eqref{normal function}.

Recall the notations from the beginning in Section 5 to \eqref{normalization tau definition}.
 When we fix a  $R$, we will use the notation $\A _R$, $d_i(R)$ and $\D _R$ instead of $\A _h$, $d_i(h)$ and $\D _h$.
    We define the section $S_t^R(0)$ of $(u_R, S_1^R)$ as we defined $S_1^h(0)$ of $(u, S_h)$  by \eqref{normalization tau definition} in Section 5. Set
    \[ d_i^R(t) = \frac{d_{i}(tR)}{d_i(R)}, \ \D_t^R=\D_{tR}(\D_R)^{-1} \text{ and } {\A}_{t}^{R}=(D_R)^{-1} \A_{tR} (\A_R)^{-1} D_{R},\]
     The normalization of $(u_R,S_1^R)$ at height $t$ is essentially the same as $(u_{tR},S_1^{tR})$:
\[	{u}_{tR}(x):= \frac{u_R( (\A_t^R)^{-1} \D_t^R  x)}{t}, \ x \in S_1^{tR}:= (\D_t^R)^{-1} \A_t^R S_1^R.\]

 Here is the main result of this section.

  \begin{Theorem} \label{uniformly strict convex thm}
	Given  large constants $R_0 >0$  and $C_0>1$. There exists a small constant $\eta >0$ such that for any  $R\geq R_0$,
	\begin{equation}\label{7.3}
		  C_0  \eta S_{R}(0)   \subset S_{\eta R}(0)  \subset (1-\eta) S_{R}(0).
	\end{equation}
Moreover, there exist $\bar{\delta}=\delta_{R_0}$ and $\gamma \in (0,1)$ such that
	\begin{equation}\label{7.4}
 c|x|^{\frac{1+\gamma}{\gamma}}-C\delta_R/\bar{\delta} \leq u_R(x) \leq C|x|^{1+\gamma}+C\delta_R/\bar{\delta}, \ \ \forall x\in S_1^R,
\end{equation}
where $\delta_R$ denotes a constant depending on $R$ and tends to $0$ as $R \to \infty$.
\end{Theorem}

\begin{Claim}\label{iteration discussion}
\eqref{7.3} is a scaling invariant property. Using iterative techniques, one can see that  \eqref {7.3} and \eqref {7.4} are equivalent.
More precisely, regardless of the term $ C\delta_R/\bar{\delta}$, the left side of \eqref{7.3} and the right side of \eqref{7.4} are equivalent to
\[ u_R(x) \leq \sigma_1(|x|)|x| \text{ for all large } R,\]
and the right side of \eqref{7.3} and the left side of \eqref{7.4} are equivalent to
\[ u_R(x) \geq \sigma_0(|x|) \text{ for all large } R, \]
where $\sigma_0, \sigma_1 $ are some universal strictly increasing $C^0$ function with $\sigma_i(0)=0$ $(i=0, 1)$.

The right side of \eqref{7.3} implies $S_{\eta^{k}} ^{\eta^{-k}R} \subset B_{(1-\eta)^k C}(0)$. Thus,  if we need  to obtain the global type estimate of $u_R$ in $ S_1^R$, we only need to prove the corresponding local estimate of $u_{\eta^{-k}R} $ in $S_{\eta^{k}} ^{\eta^{-k}R} $ for some universal large $k$. Especially,  we can obtain that  $D_nu_R$ is uniformly bounded in $S_1^R(0)$.

 By Lemma \ref{boundary nabla estimate} and the Neumann boundary value condition, the right side of \eqref{7.4} provides a universal $C^{1,\gamma}$ module for $u_R$ at $0$.
\end{Claim}

To prove Theorem \ref{uniformly strict convex thm}, in addition to the   notations  $\delta_0, c_0$   in Theorem \ref{zero petrub neumann} and
$a_R$ in \eqref{normal neumann} we need  the following  lemma.

\begin{Lemma} 
	If  $R \geq R_0$ is large, then  $a_R \geq c\delta_0 a$ and
	\begin{equation}  \label{7.5}
		cR^{\frac{1}{2}}\I'' \leq D_R'' \leq CR^{\frac{1}{2}}\I''.
	\end{equation}
\end{Lemma}
\begin{proof}
	Assume by way of contradiction that $a_R \leq  c\delta_0 a\leq \delta_0$. Then Theorem \ref{zero petrub neumann}, together with \eqref{u=0 nabla u =0} and \eqref{7.1},  implies  that the normalization solution $u_R$ is $C^{2,\alpha}$ near $0$.
This is
	\[	B_c^+(0)  \subset t^{-\frac{1}{2}}S_{t}^R(0) \subset B_C^+(0),  \;  \forall  t \in  (0, cc_0^2).\]
	We recall \eqref{7.1}
	and
	$S_{1}(0)=\D_R^{-1}\A_R S_{\frac{1}{R}}^R(0). $
	If $R \geq Cc_0^{-2}$ is large, we can take $t=\frac{1}{R} $ to obtain
	\[c\I \leq \D_R^{-1}\A_R R^{\frac{1}{2}}\leq C\I.\]
	This gives $d_i(R) \leq CR^{\frac{1}{2}}$ for $2 \leq i \leq n-1$. Hence,
	 \[	a_R =\frac{d_1(R)d_n(R)}{R} a \geq  Ca,\]
which is a contradiction.

Thus $a_R \geq c\delta_0 a$,  which implies $\Pi_{i=2}^{n-2} d_i(R) \leq C_1\delta_0^{-1} R^{\frac{n-2}{2}}$.
Noticing that $d_i(R) \geq cR^{\frac{1}{2}}$ for $2 \leq i \leq n-1$ (see \eqref{growth innitial assump}), we see that   \eqref{7.5} holds for $C=(C_1\delta_0^{-1})^{\frac{1}{n-2}}$.
\end{proof}

\begin{Remark} 
	By Lemma 7.3   we could improve \eqref{growth innitial assump} to
	\[	c|z''|^2-\bar{\delta}^{-1}  \leq u(0,z'',0) \leq  C|z''|^2,\]
	which implies
	\begin{equation} \label{7.7}
		c|z''|^2-\delta_R/\bar{\delta} \leq u_R(0,z'',0) \leq  C|z''|^2 \text{ for } R \geq \bar{\delta}.
	\end{equation}
At this moment, the above also holds when $a=0$.
\end{Remark}

{\bf  Proof of Theorem \ref{uniformly strict convex thm}:}   Due to Claim \ref{iteration discussion}, it is sufficient to prove  (7.3), which will be completed by four steps.
According to Theorem 5.1, we only need to prove (7.3) in normal direction and on tangent plane.

{\bf Step 1}.
Let $E_R= S_1^R \setminus S_{\delta_R/\bar{\delta}}^R$, we first prove that
\begin{equation}  \label{7.9}
	\begin{split}
		u_R(x)  \geq \sigma_0 (\max\{ |x_1|,|x_n|\})-C\delta_R \text{ on } E_R
	\end{split}
\end{equation}
for some universal module $\sigma_0$.

 Assume $q=(q' ,0)=(q_1, q'',0)\in \P E_R$ with $q_1\neq 0$, let
 $$\B x=x-\sum_{i=2}^{n-1}\frac{q_i}{q_1}x_1e_i $$ and consider the positive function
\[	\upsilon(x)= \frac{u(\frac{q}{2}+\frac{q_1 \B^{-1}x}{8C})}{ q_1^2}, \ x \in B_1^+ (0).\]
By the convexity, (7.5) and the assumption on $u$, we  have
\[
0 \leq \upsilon( x ) \leq  \frac{C}{ q_1^2}
\]
  and
\[
	\upsilon( x_1, x'',0  )   \leq  \frac{u(q)}{ q_1^2} +  C(|x''|^2 +\frac{\delta_R}{ q_1^2})
	  \leq C|x''|^2+\frac{u(q)+C\delta_R}{ q_1^2}.
\]
The Neumann boundary value condition guarantees
\[	D_n \upsilon( x_1, x'',0  )  \geq -C .\]
Regarding $\frac{u(q_n)+C\delta_R}{ q_1^2}$ as small constant and $q_1^{-1}$ as universal constant,  we use Lemma \ref{strict convex lemma half} to obtain
\[
\begin{split}
	\frac{u(q)+C\delta_R}{ q_1^2} \geq \sigma(\frac{C}{ q_1^2}).
\end{split}
\]
Letting $ \sigma_0(t) = t^2{\sigma}(C/t^2)$, we get
\begin{equation} \label{7.11}
	u(q',0) \geq  \hat \sigma_0(q_1) -C\delta_R, \ \ \forall  (q' ,0)\in \P E_R.
\end{equation}

Next, assume   $q=q_ne_n$ and consider the positive function
\[	v(x)= \frac{u(\frac{3q}{4}+\frac{q_nx}{8C})}{ q_n^2}, \ x \in B_1 (0).\]
Similarly, we have
\[
0 \leq v( x ) \leq  \frac{C}{ q_n^2} \text{ and } 	v( 0, x'',x_n  ) \leq C|x''|^2+\frac{u(q)+C\delta_R}{ q_n^2}.
\]
Using Lemma \ref{strict convex lemma plane},   we get
\begin{equation} \label{7.13}
	u(q_ne_n) \geq  \sigma_0(q_n) -C\delta_R
\end{equation}
for $\sigma_0(t) = t^2{\sigma}(C/t^2)$.

 Since  $S_t$ is of good shape for $t \geq \delta_R/\bar{\delta}$,
combining \eqref{7.11} and \eqref{7.13},  we have proved the desired  \eqref{7.9}.

\textbf{Step 2}. We first choose a large $R_0$   such that $\bar{\delta}R_0^{-1} $ is small enough, and define   $ \delta_R= (\bar{\delta}R^{-1}$.
Then $  \delta_R\leq \delta R_0$ for $R \geq R_0$.

For simplicity, we only consider the point $x \in E:=S_{1}^R(0) \setminus S_{\delta_R}^R(0)$ and assume $t \geq  \delta_R$  in  the rest of this Step.
By iteration, (7.6) means that there exists universal $\gamma \in (0,1)$ such that
\[   u_R(x) \geq c \max \{ |x_1|,|x_n|\}^{\frac{1+\gamma}{\gamma}},\]
which is
\[	\max\{d_1^R(t),d_n^R(t)\} \leq  Ct^{\frac{\gamma}{1+\gamma }}.\]
Recalling (7.4)   and   $\Pi_{i=1}^n d_i^R(t) =t^{\frac{n}{2}}$,  we have
\[	\min\{d_1^R(t),d_n^R(t)\} \geq  ct^{\frac{1}{1+\gamma }}.\]
 This gives (7.3) in the normal direction.

\textbf{Step 3}.  We now have
 \begin{equation}\label{d1 control}
	ct^{\frac{1}{1+\gamma }} \leq d_1^R(t) \leq Ct^{\frac{\gamma}{1+\gamma }},
\end{equation}
  whose right side gives
 \[	u_R (x',0) \geq  c|x_1|^{\frac{1+\gamma}{\gamma }}.\]
  Noticing (7.5) and  that $u_R$ is locally Lipschitz, we get
\begin{equation}\label{d12 mix low barrier 1}
	u_R (x',0) \geq \max\{ c|x''|^2-\frac{C\delta_R}{\delta}|x_1| ,0\}.
\end{equation}
Let $E_1 =\{x'\P(E):\; |x_1| \leq \frac{c^2}{4}|x''|^2\} $ and $E_2= \P(E)\setminus E_1$. The above two inequalities implies  that
\begin{align*}
	u_R (x',0) & \geq \max\{ c|x''|^2-C|x_1| ,0\} \chi_{E_1} +  c|x_1|^{\frac{1+\gamma}{\gamma }}\chi_{E_2}\\
	& \geq \frac{c}{2}|x''|^2\chi_{E_1} +  c|x_1|^{\frac{1+\gamma}{\gamma }}\chi_{E_2} \\
	& \geq c|x'|^{\frac{2(1+\gamma)}{\gamma }},
\end{align*}
Which proves the right side of (7.3) on the  tangent plane.

\textbf{Step 4}.
Given constant $M$, we claim that there exist universal $ \varepsilon $,  such that
\[  \sup\{ s|\, se_1 \in G_{t}^R(0) \} \geq M \varepsilon  \]  for some   $t \in [ \varepsilon , \frac{1}{2}]$.
Then This claim,  (7.5) and (5.8) implies that
\[\varepsilon^{-1}G_\varepsilon^R(0) \supset t^{-1}G_t^R(0) \supset cMB_1(0) \supset cM\P(S_1^R(0)),\]
 which proves the left side of (7.3) on the  tangent plane.

On the contrary to this claim, by the balance property of $G_{t}^R(0)$ we get
\[   u_R(x_1,0) \geq M^{-1}|x_1| \text{ for } x_1 \geq CM\varepsilon.\]
For simplicity, we may require the point $x \in S_{1}^R(0) \setminus (S_{CM\varepsilon}^R(0) \cap \{ |x_1| \leq  2CM\varepsilon \})$ and assume constant $t \geq  CM\varepsilon$
in  the rest of this Step.  Then by the convexity and (7.5) we have
\[ u_R(x_1,x'',0) \geq 2u_R(2x_1,0,0 ) - u_R(0,-x'',0 )  \geq 2M^{-1}|x_1| -C|x''|^2. \]
This together with (7.10) implies that
\[   G_t^R(0) \subset \{  c(|x''|^2 -t) \leq |x_1 | \leq CM(|x''|^2 +t)\}.\]
Since $0 \in G_t^R(0) $, the above relation then implies that
\[ G_t^R(0) \subset  \{ |x_1| \leq CM^4t\}.\]
Thus, $d_1^R(t) \leq CM^4t$.  Taking $t=\varepsilon $ is small enough, this contradicts (7.9). In this way,
 we completes the  proof of Theorem 7.1.

\section{Liouville Theorem}
In this section, we will modify  the techniques developed from the classical $C^{1,\alpha}$ estimate \cite{[C3]} to prove   Theorem \ref{liouville theorem}.
We always suppose that $u$ is a viscosity solution to problem \eqref{liouville problem} and the hypothesis of Theorem \ref{liouville theorem} is satisfied. As we have said, we may assume
\eqref{u=0 nabla u =0}.
Recall the normalization family $(u_R, S_1^R)$ defined by \eqref{normal function}, which satisfies \eqref{normal section}-\eqref{normal genral assump}.

 Consider the function
\begin{equation}\label{8.1}
	v_R(x)=D_nu_R(x)-a_Rx_1,
\end{equation} and let
\begin{equation} \label{8.2}
	\begin{split}
		& m_R =\inf\{\frac{v_R(x)}{x_n}|\, x\in S_1^R(0)\}, \\
		 & M_R =\sup\{\frac{v_R(x)}{x_n}|\, x\in S_1^R(0)\},\\
	 &  \omega_R =\frac{M_R}{m_R}-1.
\end{split}
\end{equation}
Our goal is to prove that
\begin{equation}  \label{8.3}
	c \leq m_R \leq M_R \leq C
\end{equation}
and $\omega_R$ is strictly increasing in $R$. For this purpose, we need  the auxiliary function like those
from \eqref{subtract support function} to \eqref{sup sub test function}.   Given $y\in  G_{\frac{1}{2}}^R(0)$,  let
 \begin{equation}  \label{8.4} u^{R,y}(x)= u_R(x) - u_R(y)-\nabla u_R(y) \cdot (y-x)
 \end{equation}
 and define the functions
\begin{equation}  \label{8.5}
	\begin{split}
		 & w^{R,y}(x) =[u^{R,y}(x)-\frac{n}{2}x_nD_nu^{R,y}(x)],\\
		& 	w^{R,y,+}(x) =C_1[	w^{R,y}(x) +C_2x_n],\\
		& 	w^{R,y,-} (x) =c_1[x_n-c_2	w^{R,y}(x) ].
	\end{split}
\end{equation}
In order to show that  $w^{R,y,\pm}$ are supersolution/subsolution to the following linearized  problem  \eqref{8.17}for $v_R$,  we need the following two lemmas.

\begin{Lemma} 
\label{strict quali line lemma}
	Suppose that  $u$ is a viscosity solution to the first equation of problem (1.1) and  (2.6) is satisfied. Assume that
	\begin{equation} \label{8.6}
		B_c^+(0) \subset S_1^u(0) \subset B_C^+(0)\; \text{and}\;  ||u||_{Lip(B_C(0))} \leq C.
	\end{equation}
	  Define $h_x^u = \sup\{ h |\, S_{t}^u(x) \subset B_1^+(0),\ \forall t < h\}$.
	There exists small universal constants $\delta  >0$, $\rho >0$ such that if
	\begin{equation}\label{8.7}
		c|x|^{\frac{1+\gamma}{\gamma}}-\delta\leq u(x) \leq C|x|^{1+\gamma}+\delta, \ \ \forall x\in S_1^u(0),
	\end{equation}
	then
	\[ h_x^u \geq \sigma_0(x_n) \text{ in } B_{\rho}^+(0).\]
\end{Lemma}
\begin{proof}
	If the result were false, we could find  a constant $\tau >0$, a sequence of  convex  viscosity solutions $u_k$
satisfying (8.6)-(8.7)  and points $p_k \in B_{\rho}(0)$
	such that
	\[p_k \cdot e_n \geq \tau,  \ \  \text{and}\;
	h_{p_k}^{u_k} \to 0.\]
	Letting $k \to \infty$, we get a limit $v$  which still satisfies our assumption and $p_k$ converges to a point $p$ such that  $h_{p}^{v}=0$.
Let $\Sigma :=S_0^v(p)$. Then $\Sigma $ is not a single point set.  (8.7)    implies that $ \Sigma \subset B_{\delta_1}(0)$ for some constant $\delta_1>0 $
depending on $\delta , \rho$.  Hence we can find an interior extremal point
	\[q\in \{ y|\, y\in \Sigma , y_n  =\sup_{x\in \Sigma}x_n\} ,\]
	  which contradicts  Lemma  2.14.
\end{proof}

\begin{Lemma}
	 For large $R \geq R_0$, one has
	\begin{equation}\label{8.8}
		\sigma_0(x_n) \leq   v_R(x)  \leq C \text{ in } S_1^R(0).
	\end{equation}
\end{Lemma}

\begin{proof}
As we have discussed in Claim \ref{iteration discussion}, we need only to prove  (8.8) for $ y \in S_{\sigma_0(|c_0|)}^R(0)\subset B_{c_0}(0)$
 for a small universal $c_0$.  According to Lemma 2.5,  the Neumann boundary value condition and \eqref{7.4} provides
 a $C^1$ module for $u_R$ at $0$.  This is,
 \begin{equation}\label{8.9}|\nabla u(x)| \leq C|x|^{\gamma}, \; \text{for}\; x \; \text{near}\;  0,
 \end{equation}
  which implies
	  the right side of \eqref{8.8}.

As in Lemma 8.1 we denote
\begin{equation} \label{8.10}
	h_y^R = \sup\{ h |\, S_{t}^{R}(y) \subset \R_+^n,\ \forall t < h\}.
\end{equation}
We will find  a universal $C^0$ module  $\sigma_0$ such that   the following three inequalities hold in $ S_1^R(0) $:
\begin{equation}  \label{8.11}
	h_x^R \geq \sigma_0(x_n),
\end{equation}
\begin{equation}  \label{8.12}
	\sigma_0(x_n) \I \leq   D^2u_R(x) \leq \frac{1}{ \sigma_0(x_n)} \I,
\end{equation}
and
\begin{equation}  \label{8.13}
	v_R(x)\geq \sigma_0(x_n).
\end{equation}

 Again,  it is sufficient to prove   (8.10)-(8.12) for $x$ around $0$.
	Suppose $c_0>0$ is small and $y\in  S_{ c_0} ^R(0)$, then  (7.3) and (8.9) means
	\begin{equation*}
		|\nabla u_R(y)|+u_R(y) \leq c_1:=c_0+Cc_0^{\gamma}
	\end{equation*}
	and $c_1$ is still small. Thus, the section $S_{h_y^R}^R(y) \cap B_C(0) $ is contained in  $S_{c_1}^R(0) \subset S_{\frac{1}{2}}^R(0)$
and will touch $\partial \R_+^{n}$ at some point $z=(z',0) \in G_{c_1}^R(0)$.
	
Let $l_{zy}$ be the line segment connecting $z$ and $y$. Extend   $l_{zy}$  such that it intersects $\partial B_{{c}}^+(0)$ at point $q$.
	Set
	 $$w(x)=u_R(x)-u_R(y)-\nabla u_R(y)\cdot (x-y).$$ Then
	\begin{equation}  \label{8.14}
		w(x) \leq h_y^R, \ x \in l_{yz}.
	\end{equation}
Furthermore, it follows from Theorem 7.1 and Lemma 8.1 that
 \begin{equation}  \label{8.15}
		h_y^R \geq \sigma_0(y_n)
	\end{equation}
	for $\sigma_0(t) :=ct^{\frac{2}{n}} \sigma(c t^{-\frac{2}{n}})$.
	
	The classical  estimate  result (see Lemma \ref{volume bound by measure}) gives $\operatorname{Vol}(_{h_y^R}^R(y)) \sim (h_y^R)^{\frac{n}{2}}$.
 Notice that $S_{h_y^R} ^R \subset B_C^+(0)$, thus $S_{h_y^R}^R\supset B^+_{c(h_y^R)^{\frac{1}{2}}}(0) $
and the classical interior estimate (See Lemma \ref{c2a classical estimate}  and its proof) then gives
	\begin{equation}  \label{8.16}
		\hat \sigma_0(y_n) \I \leq   D^2u_R(y) \leq \frac{1}{\hat \sigma_0(y_n)} \I
	\end{equation}
	for $\hat \sigma_0(y_n)=\min\{ [\sigma_0(y_n)]^{n-2}, \sigma_0(y_n)\}$.
	
Now, up to a bounded linear transform which maps  in $S_{\sigma_0}^R(0)$
to   $S_1^{\sigma_0 R}$, (8.15)-(8.16) and
the statement in Claim \ref{iteration discussion} implies that (8.11)  and (8.12)    in $S_1^R(0) $.
	
	Finally, let us prove the (8.13). Similarly as  Lemma \ref{linearized problem}, the non-negative function $v_R (x)$ in \eqref{8.1}  solves the linearized problem:
	\begin{equation}  \label{8.17}
		\begin{cases}
			U_R^{ij} D_{ij}v_R=0        & \text{ in } S_1^R(0)\\
			0\leq  v_R \leq C    & \text{ in } S_1^R(0) \\
			v_R=  0                      & \text{ on }  {G_1^R(0)}
		\end{cases}.
	\end{equation}
Notice that we also have $v_R(ce_n) \geq c $ by the convexity. Due to (8.12),  the  first equation in (8.17) is
uniformly elliptic with norm of the elliptic coefficients depending on $\sigma_0(y_n)$
in the domain $E(y_n)=S_1^R(0) \cap \{ x_n \geq \frac{y_n}{2}\}$. We apply the classical Harnack inequality  to obtain
	\begin{equation}\label{8.18}
		v_R(y',y_n) \geq C(||U_R^{ij}||_{E(y_n)},||U_R^{ij,-1}||_{E(y_n)}, \frac{y_n}{2})  v_R(ce_n) \geq  \sigma_0(y_n)
	\end{equation}
  in $ S_{\frac{1}{2}} \cap \{x_n \geq \frac{y_n}{2}\}$ for some new $\sigma_0$. Thus, the proof is completed.
	
\end{proof}

\textbf{Proof of Theorem \ref{liouville theorem}:} It follows from
Lemma 8.2  that  $v_R (x)$ satisfies the following problem:
\begin{equation} \label{8.19}
	\begin{cases}
		U_R^{ij} D_{ij}v_R=0        & \text{ in } S_1^R(0)\\
		\sigma_0(x_n) \leq  v_R \leq C    & \text{ on } \partial S_1^R(0) \setminus G_1^R(0)  \\
		v_R=  0                      & \text{ on }  {G_1^R(0)}
	\end{cases}.
\end{equation}
Then as Corollary \ref{lip unversal by sup} and 3.8 we see that   the function $	w^{R,y,\pm}(x)$ in \eqref{8.5}
 are supersolution/subsolution of  problem  (8.19) in $S_c^R(0)$,
and $cx_n \leq v_R (x) \leq Cx_n$  in $ S_c^R(0)$.
Using  the definition (8.2), we get
\begin{equation*}
	c \leq m_R \leq M_R \leq C,
\end{equation*}
which is the desired (8.3).

Next, we are going to prove
\begin{equation} \label{8.20}
	\omega_{\rho R} \leq (1-\theta)\omega_{ R}
\end{equation}
for some universal $\rho >0$ and $\theta\in (0, 1)$.

 Let
 \begin{equation*}	\begin{split}
	& m_{R,t} =\inf\{\frac{ v_R(x)}{x_n}|\, x\in S_t^R(0)\} ,\\
	& M_{R,t} =\sup\{\frac{ v_R(x)}{x_n} |\, x\in S_t^R(0)\}, \\
 	& \omega_{R,t} =\frac{M_{R,t}}{m_{R,t}}-1.
 \end{split}
\end{equation*}
  It is clear that for all  $ t\in (0, 1)$,
\[		m_{R,1} \leq m_{R,t}  \leq  M_{R,t} \leq M_{R,1}  \]
and
\[\omega_{R,t} =\omega_{Rt}.\]
To prove (8.20), it is sufficient to  prove
\begin{equation} \label {8.21}
	\omega_{R,\rho } \leq (1-\theta)\omega_{ R,1}
\end{equation}
for some universal $\rho >0$ and $\theta\in (0, 1)$.
 Assume $\omega_{ R,1}>0$ and consider the two non-negative functions:
\begin{equation} \label{8.22}
	\begin{split}
		H_1(x) =\frac{v_R(x)-m_{R,1}x_n}{M_{R,1}-m_{R,1}}, \\
		H_2(x) =\frac{M_{R,1}x_n-v_R(x)}{M_{R,1}-m_{R,1}}.
	\end{split}
\end{equation}  It follows from (8.19) that
each $H_i $ $ (i=1, 2)$ solves
\begin{equation} \label{8.23}
	\begin{cases}
		U_R^{ij} D_{ij} H_i=0        & \text{ in } S_1^{R}(0)\\
		0 \leq  H_i \leq 1    & \text{ in } S_1^{R}(0) \\
		H_i=  0                      & \text{ on }  {G_1^R(0)}
	\end{cases}.
\end{equation}
While $H_1+H_2=x_n$, we  may assume $H_1(ce_n) \geq \frac{c}{2}$ without loss of generality.

Since (8.12)   means  that  the  first equation in (8.23)  is universal elliptic away from $G_1^R(0)$, as (8.18) we  have
\begin{equation} \label{8.24}
	\begin{split}
		H_1(x)  \geq \sigma_0(x_n)H_1(ce_n) \text{ in } B_{c/2}^+(0).
	\end{split}
\end{equation}

Again, for any $y=(y',0) \in G_{c_0}^R(0)$ with $c_0>0$  small enough, we could find  a universal
constant $\bar{c}>0$ such that $B_{c(\bar{c})/2}^+(0) \subset S_{\bar{c}}^{R}(y) \subset B_{c/2}^+(0) $.

With some universal constant $C_1, C_2, c_1, c_2$  depending  on $\bar{c}$ and $\sigma_0(x_n)$, the function  $	w^{R,y,\pm}(x)$
in \eqref{8.5}   is still   supersolution/subsolution of $H_1$   to the following  problem:
 \begin{equation} \label{8.25}
\begin{cases}
U_R^{ij} D_{ij}H_1=0        & \text{ in } S_{\bar{c}}^{R}(y)\\
 \sigma_0(x_n) \leq  H_1 \leq   1    & \text{ on } \partial S_{\bar{c}}^{R}(y) \setminus G_{\bar{c}}^{R}(y)  \\
H_1=  0     & \text{ on }  G_{\bar{c}}^{R}(y)
\end{cases}.
\end{equation}
In conclusion, we have
 \begin{equation} \label{8.26}
 H_1(y+te_n) \geq w^{R,y,-}(y+te_n) \geq \frac{c_1}{2} t \; \text{ for } \;  t \leq t_0.
\end{equation}

On the other hand, it follows from (7.2) (or (7.3)) that
  \begin{equation} \label{8.27}  S_{\rho}^R(0) \subset G_{c_0}^R(0) \times [0,t_0) \; \text{ for} \;  \rho =c\min\{ c_0^{\frac{(1+\gamma)^2}{\gamma}}, t_0\}.
 \end{equation}
 Combining (8.26) and (8.27),  we have
 \begin{equation*}
\frac{v_R(x)-m_{R,1}x_n}{M_{R,1}-m_{R,1}}\geq  \theta x_n  \;  \text{ in } \;  S_{\rho}^R(0)
\end{equation*}   for $\theta=\frac{cc_1}{2} $.
That is
\[m_{R,\rho}\geq m_{R,1}+\theta ( M_{R,1}-m_{R,1}) \;  \text{ in }  \; S_{\rho}^R(0) ,\]
 which yields the desired (8.21). In this way, (8.20) has been  proved.


By (8.3) and (8.20) we see that
\begin{equation*}
\omega_{R} =\omega_{\rho^k \rho^{-k}R}  \leq (1-\theta)^k\omega_{ R\rho^{-k}}  \leq C(1-\theta)^k.
\end{equation*}
Letting  $k \to 0$, we get $\omega_{R}=0$ for any $R$.
Hence, $v_R(x)=D_nu_R-a_Rx_1$ is a linear function in $S_R(0)$. Then letting $R\to \infty$, we get the limit function
 $\hat v(x):=D_nu(x)-ax_1=cx_n$ for some $c>0$.

Lemma \ref{strict convex in interior} means that $u$ is strict convex and smooth away from $\partial \R_+^{n}$. Thus
 $$u(x) = f(x')+ax_1x_n+cx_n^2$$  for some smooth function $f$ in $\R^{n-1}$. Let $$g(x')=f(x')-\frac{a^2}{4c}x_1^2.$$
 Since $D^2 u(x',1) > 0$, we have  $D^2g(x') >0$. Therefore,  $g$ is a convex, smooth  function solving
\begin{equation*}
\operatorname{det} D^2 g =1\;  \text{ in }\;  \R^{n-1}.
\end{equation*}
The classical results (Theorem \ref{global classical results}) shows that $g$ is quadratic polynomial of $x'$, which implies that $u$ is a quadratic polynomial of $x$.
The proof of Theorem \ref{liouville theorem} is completed.

\begin{Remark} 
	In the case $n=2$,  there are two other proofs for   Theorem \ref{liouville theorem}, which  depends on the strict convex Lemma \ref{strict convex lemma half} and \ref{strict convex lemma plane}.
 With the two lemmas, one can show the solution is $C^1$ up to boundary and is of  super-linearity  growth at infinite. This fact,    together
 with partial Legendre transform and the Liouville theorem for harmonic functions, will gives us the first proof. As for the second proof, we
 need to use Wang's function\cite{[W95]}  to construct supersolution/subsolution and so prove that the $u$ is $C^{1,1}$ up to boundary, then
by blowing up at infinity to obtain $D^2u=c\I$.
\end{Remark}
As the end of this paper, we give the two proofs in more details.
	
	\textbf{Proof 1.}
	 We write a point in $\R^2$ as $(x, y)$ and use the classical partial Legendre transform in the $x$-variable. For any $p \in R$,   define
	\[ u^*\left(p,y\right):= \sup_{x \in R}\{ px-u\left(x,y\right)\}. \]
The uniformly strict convexity of $u$ at boundary and the smoothness of $u$ at interior also implies the superlinearity at infinite,  which guarantees
 the uniqueness and existence of $x= X\left(p,y\right) $ such that the above supremum is attained.
	
	By a calculation the two dimensional case of problem \eqref{liouville problem},
	\[\operatorname{det}  D^2u=1 \;  in \; \R_+^2,\ \    D_2 u\left(x,0\right)= ax  \;  on \;  R\]
	is  turned to  the problem for $u^*$:
	$$\begin{cases}
		\Delta u^*=1          & \text{ in } \R_+^2\\
		D_{v^*}u^*= 0                      & \text{ on } \R
	\end{cases}, $$
	where $v^*=e_2+ae_1$. Differentiating $u^*$  with respect to $v^*$ and letting $V=D_{v^*}u^*$ we obtain
	$$\begin{cases}
		\Delta V=  0        & \text{ in } \R_+^2\\
		 V= 0                      & \text{ on } \R
	\end{cases}. $$
	Note that
$$V=D_{v^*}u^*\left(p,y\right)=aX\left(p,y\right)-u_2\left(X,y\right) = u_2\left(X,0\right)-u_2\left(X,y\right) \leq 0.$$
 We can use the Liuoville theorem for harmonic functions to see that the only solutions to above problem are
 $$V(x, y)=D_{v^*}u^*\left(x,y\right) \equiv A$$    for some constant $A$.  Then
	\[ u^*\left(p,y\right)=g\left(p-ay\right)+ Ay \]
	for some function $g:[0,\infty ) \to R $.  Recalling $\Delta u^*=1, D_{v^*}u^*= 0 $,  we have
	\[g\left(0\right)=0 \ and \ \left(1+a^2\right)g''=1 , \]
which implies that  $u^*$ is a quadratic  polynomial, and so is  the $u$.
	
	\textbf{Proof 2. }
	As we have said,  when $n=2$, one could use strict convexity Lemma \ref{strict convex lemma half} and \ref{strict convex lemma plane} to show
that the solution  $u$ to problem \eqref{liouville problem} is uniformly $C^{1,\alpha}$ up to boundary\cite{[C3]}.
	
	We consider Wang's singular function with bounded positive Monge-Amp\'epre measure  for large $a>>2$:
	\[
	W_{a,b}\left(p,q\right)=\begin{cases}  p^{a}+\frac{a^2-1}{a(a-2)}p^{a-\frac{2a}{b}}q^2,\ |p|^a \geq |q|^b \\
		\frac{4a-5}{2(a-2)}q^{b}+\frac{1}{2a}q^{b-\frac{2b}{a}}p^2,\ |q|^b \geq |p|^a
	\end{cases}.
	\]
	Let
$$W\left(x_1,x_2\right)= Ca W_{a,b}\left(\tau x_1,x_2/ \tau \right),$$
 where $\tau >0 $ is small and satisfies $a\tau^a  \leq c$.
	One can use this function to construct a supersolution $(W^+ )$/subsolution  $(W^- )$at boundary point:
	\[ \begin{cases}w^+\left(x\right) = u\left(x\right)-  W\left(x_1,x_2\right) +Mx_2 \\
	  w^-\left(x\right) = \epsilon_1 \left(W\left(x_1,x_2\right)-u\left(x\right)+\epsilon_2 x_2\right)
\end{cases}.   \]
	This will give the uniform estimate $c \leq D_{nn}u \leq C$ on the boundary. Therefore $u$ is $C^{1,1}$ and then in $C^{2,\alpha}$.
After blowing up at infinite, this together with the superlinearity will  imply $D^2u=c\I$ and the proof is completed.

\end{document}